\documentclass[11pt,a4paper]{article}
\usepackage[T1]{fontenc}
\usepackage{lmodern,amsmath,amssymb,amsthm,mathtools}
\usepackage[margin=28mm]{geometry}
\usepackage{microtype,enumitem,array,booktabs,mathrsfs,esint}
\usepackage[hidelinks]{hyperref}
\hypersetup{pdftitle={Complete conformal metrics of constant positive scalar curvature with prescribed isolated singularities}}
\numberwithin{equation}{section}
\newtheorem{theorem}{Theorem}[section]
\newtheorem{proposition}[theorem]{Proposition}
\newtheorem{lemma}[theorem]{Lemma}

\theoremstyle{definition}

\theoremstyle{remark}

\DeclareMathOperator{\diver}{div}
\DeclareMathOperator{\tr}{tr}
\DeclareMathOperator{\Ric}{Ric}

\DeclareMathOperator{\sech}{sech}
\DeclareMathOperator{\dist}{dist}
\newcommand{\R}{\mathbb R}
\newcommand{\Sph}{\mathbb S}
\newcommand{\eps}{\varepsilon}
\newcommand{\scnorm}{\mathrm{sc}}

\allowdisplaybreaks[2]
\setlist{nosep}
\title{Complete conformal metrics of constant positive scalar curvature with prescribed isolated singularities}
\author{%
Hongyu Chen\thanks{School of Mathematics, Sichuan University,
Chengdu 610064, China.
Email: \texttt{hongyu.chern@gmail.com}.}
\qquad
Li Ma\thanks{Corresponding author.
School of Mathematics and Physics,
University of Science and Technology Beijing,
Beijing 100083, China.
Email: \texttt{nuslma@126.com}.}
}
\date{Oct. 20th, 2026}
\begin{document}
\maketitle
\begin{abstract}
Let $(M^n,g)$ be a smooth closed connected Riemannian manifold of dimension $n\ge3$ with positive Yamabe constant. For any prescribed finite set $\Lambda\subset M$ containing at least two points, we construct a complete metric on $M\setminus\Lambda$, conformal to $g$, with scalar curvature $n(n-1)$. At each singularity, the local geometric term in the reduced action comes from the first nonzero homogeneous metric term of degree at most $(n-2)/2$ in finite conformal normal coordinates, or from the mass of the Green function metric when all such terms vanish. We completely solve Schoen's conjecture about singular Yamabe problem.
\end{abstract}

\section{Introduction}
\textbf{This work we completely solves the following conjecture posed by R.Schoen about the singular Yamabe problem in all dimensions.}
\textbf{In our previous work \cite{CMS26}, we can solve Schoen's conjecture in dimensions from three to seven.}

Let $(M^n,g)$ be a smooth closed connected Riemannian manifold of dimension $n\ge3$, and let $\Lambda\subset M$ be finite. The positive singular Yamabe problem with singular set $\Lambda$ asks for a complete metric on $M\setminus\Lambda$, conformal to $g$, with scalar curvature $n(n-1)$. Writing this metric as $u^{4/(n-2)}g$ and using the convention $\Delta_g=\diver_g\nabla$, we seek a positive smooth solution of
\begin{equation}\label{eq:yamabe}
-\frac{4(n-1)}{n-2}\Delta_g u+R_gu=n(n-1)u^{(n+2)/(n-2)}.
\end{equation}

The compact Yamabe problem was initiated by Yamabe \cite{Yamabe} and resolved through the work of Trudinger \cite{Trudinger}, Aubin \cite{AubinYamabe} and Schoen \cite{SchoenYamabe}; Lee and Parker \cite{LP} give an account of this theory. The Euclidean bubbles used below are extremals for the sharp Sobolev inequality studied by Aubin \cite{AubinSobolev} and Talenti \cite{Talenti}.

On the round sphere, Schoen \cite{Schoen} proved existence for every finite singular set containing at least two points. Mazzeo and Pacard \cite[Theorem 1]{MP} later gave a gluing construction of these metrics. Positive-dimensional singular sets were also studied: Mazzeo and Smale \cite{MS} constructed metrics for small perturbations of equatorial subspheres of dimension $0<d<(n-2)/2$, and Mazzeo and Pacard \cite{MP96} treated suitable smooth singular submanifolds. More recently, Li and Zhang \cite{LiZhang} proved existence on the complement of any closed set $\Lambda\subset\Sph^n$ containing at least two points and satisfying $\dim_{\rm H}\Lambda<(n-2)/2$.

The radial models belong to the classical Emden--Fowler theory \cite{Fowler}. For the flat critical equation, Caffarelli, Gidas and Spruck \cite{CGS} established asymptotic radial symmetry near a nonremovable isolated singularity. Korevaar, Mazzeo, Pacard and Schoen \cite{KMPS} refined the expansion using conformally deformed radial solutions and related the asymptotics to Pohozaev invariants. Han and Li \cite{HanLi} constructed local solutions with prescribed asymptotic expansions to arbitrarily high order. On general smooth backgrounds, Fowler asymptotics were proved by Marques \cite{Marques} in dimensions three through five and by Xiong and Zhang \cite{XiongZhang} in dimension six. Han, Xiong and Zhang \cite{HXZ} obtained higher-dimensional results under quantitative flatness and growth hypotheses. Han, Jiang, Wang, Xiong and Zhang \cite{HigherI} subsequently proved the removable-or-Fowler alternative for arbitrary smooth backgrounds in dimensions $3\le n\le29$. Their second paper \cite{HigherII} constructs complete singular metrics with non-Fowler behavior in dimensions $n\ge30$.

The analysis of the ends also enters deformation and gluing theory. Mazzeo, Pollack and Uhlenbeck \cite{MPUModuli} developed the Fredholm theory for asymptotically periodic ends and described the local real-analytic structure of the moduli space. Pollack \cite{Pollack} proved compactness for fixed punctures under a uniform nondegeneration bound on the dilational Pohozaev invariants. Connected-sum constructions were developed by Mazzeo, Pollack and Uhlenbeck \cite{MPUConnected}, and Ratzkin \cite{Ratzkin} glued compatible Delaunay-type ends. Chru\'sciel and Pollack \cite{CP} deformed asymptotically Delaunay ends to exactly Delaunay ends. Chru\'sciel, Pacard and Pollack \cite{CPP} also constructed such ends by localized deformations of generic metrics of constant positive scalar curvature. The latter deformations allow the conformal class to change.

For a prescribed compact background, Byde \cite[Theorem 1.1]{Byde} attached Delaunay ends at locally conformally flat points under nondegeneracy of the positive constant-scalar-curvature solution. Section~6 of the same paper gives a different construction about the zero solution, using a small positive sum of Green functions. Positive Yamabe constant makes that linearization invertible. If all prescribed points have locally conformally flat neighborhoods, choose a conformal representative that is flat near them. The diagonal entries of its Green interaction matrix are nonnegative by positive mass, while the off-diagonal entries are strictly positive. With at least two points, equal positive weights therefore make every row sum positive, so all the ends in Byde's construction are complete. Together with the positive mass theorem used in Lemma~\ref{lem:mass}, this recovers the locally conformally flat case of the theorem below.

For a compact metric $g$ with scalar curvature $n(n-1)$, Santos \cite[Main Theorem and Theorem 1.5]{Santos} constructs Delaunay-type ends under the assumption that $\Delta_g+n$ is invertible. In dimensions $n\ge6$, the construction also requires $\nabla^jW_g(p)=0$ at each prescribed point for $0\le j\le\lfloor(n-6)/2\rfloor$, where $W_g$ is the Weyl tensor. The exterior construction is a perturbation of this fixed background metric.
In the Introduction to \cite{HigherII}, Han, Jiang, Wang, Xiong and Zhang announce, as forthcoming work, complete metrics of constant positive scalar curvature in every positive Yamabe conformal class not conformally equivalent to the round sphere, with any prescribed nonempty finite singular set and Fowler asymptotics at every point.

 In this paper, we show the below.
\begin{theorem}\label{thm:main}
Let $(M^n,g)$ be a smooth closed connected Riemannian manifold with $n\ge3$ and $Y(M,[g])>0$. Let $\Lambda=\{p_1,\ldots,p_k\}$ be any set of distinct points, with $k\ge2$. There exists a function $u\in C^\infty(M\setminus\Lambda)$, $u>0$, satisfying \eqref{eq:yamabe}, such that $(M\setminus\Lambda,u^{4/(n-2)}g)$ is geodesically complete.
\end{theorem}

 \textbf{Using the gluing method, (see \cite{Schoen}, also Byde \cite[Section~6]{Byde}), we begin the exterior construction with a positive sum of Green functions whose coefficients tend to zero.}
 The limiting linear operator away from $\Lambda$ is therefore the conformal Laplacian, which is positive under $Y(M,[g])>0$. Solving the local source equations, the exterior equation and the physical Cauchy matching together gives a positive $C^2$ family whose residual lies in $n+1$ prescribed directions at each puncture. The remaining finite-dimensional equations set the residual coefficients to zero.

The leading terms in these equations come from the geometry at each puncture. After a finite conformal normalization, suppose that the first nonzero homogeneous metric term has degree $2\le d\le(n-2)/2$. We compute its quadratic contribution to the action after including the linear response of the conformal factor, using the single-degree form of \cite[Sections~6.4--6.5]{HigherI} and the critical diagonal logarithmic term of \cite[Appendix~A]{KMS} in our action normalization. For $d<(n-2)/2$, the corresponding integral over the Euclidean bubble is finite and strictly negative. At $d=(n-2)/2$, the integral has a logarithmic divergence whose coefficient is strictly negative for every nonzero homogeneous tensor allowed by the normalization. The calculation concerns one homogeneous degree; the quadratic forms used in the compactness theory for smooth Yamabe solutions also contain interactions between different degrees \cite{KMS}.

If all homogeneous terms through degree $(n-2)/2$ vanish, the Green function metric is complete, scalar flat and asymptotically flat. Positive mass entered the compact Yamabe problem through Schoen's argument \cite{SchoenYamabe}; the classical proofs are due to Schoen and Yau \cite{SY} and Witten \cite{Witten}. Here its mass is strictly positive outside the round conformal class, by the positive mass theorem and rigidity statement of Khuri, Wang and Wang \cite[Theorem 1.1]{KWW}. Lemma~\ref{lem:mass} verifies the decay required to apply that result. These masses and the pairwise Green function values give the leading interactions for the corresponding points.

To transfer the model coefficients to the constructed family, we parametrize each end by its phase and center at the matching interface. Their terminal values are solved for as functions of the nonlinear source. The static bubble and its chart then remain fixed during source differentiation. We retain the complete first-order geometric correction on both sides of the interface and obtain joint second-derivative estimates for the matching problem.

We renormalize the action at the actual terminal phases and combine the interior and exterior physical boundary terms before passing to the limit. A critical point of this action eliminates every residual coefficient. To find such a point, we control first and mixed second derivatives at the scale of each individual end. The leading mass-scale function is strictly concave. Together with the nondegenerate curvature-scale and mass-center equations, this permits an implicit elimination of those variables. The center dependence of the leading geometric terms then gives a strict boundary gap, so the reduced action attains an interior minimum. The separate error scales allow this argument to handle points with different leading degrees together with mass points.

The dimension, background metric and prescribed point set are fixed throughout. Constants may depend on these data and on the compact parameter sets specified in the proof.

The proof of Theorem \ref{thm:main} goes in the following way. We study the Green functions in Section \ref{sect2} and Lemma~\ref{lem:mass} supplies strictly positive mass for its mass points. In Section \ref{sect3} we study the radial equations and study its low spherical modes. High modes are studied in Section \ref{sec:ends} and  we establish the key result,
Proposition~\ref{prop:end} with local smallness conditions and in Section \ref{sec:projected}, we make the energy choices \eqref{eq:energy-choice} satisfy the local smallness conditions. In Sections~\ref{sec:inverse}--\ref{sec:projected} we produce
 the common positive projected family $(U)$. Then we do computation in
 Sections~\ref{sec:coefficients}--\ref{sec:action} and make the finite-dimensional argument to give parameters for which it solves $P_gU=QU^p$. We complete the proof in Section \ref{sec:reduction}.

\section{Conformal normalization and the Green function}\label{sect2}

Put $m=n-2$, $a=m/2$, $p=(n+2)/(n-2)$, $c=m/[4(n-1)]$, $Q=n(n-2)/4=a(a+1)$, and $\omega=|\Sph^{n-1}|$. We write
$$
P_g=-\Delta_g+cR_g,\qquad \mathcal N_g u=-P_gu+Qu^p.
$$
Thus \eqref{eq:yamabe} is $P_gu=Qu^p$. Fix a H\"older exponent $0<\alpha<1$. Unless stated otherwise, $C$ denotes a positive constant that may change from line to line and is independent of the small scales once the fixed choices have been made. We normalize the Green functions by
$$
P_gG_i=m\omega\delta_{p_i}.
$$

Apply the finite conformal normal coordinate construction of \cite[Theorems 5.1--5.2 and Lemma 5.3]{LP} in disjoint neighborhoods of the prescribed points. Extend the logarithms of the conformal factors by cutoffs and make the resulting global conformal change. The new metric, still denoted by $g$, remains fixed until the final step. Choose the finite normalization order $N$ larger than all derivative and remainder orders required below.

In these coordinates,
\begin{equation}\label{eq:finite-normal}
\begin{gathered}
g=e^B+E_N,\quad Bx=0,\quad\tr B=0,\quad E_N=O_{\scnorm}(r^N),\\
gx=x,\qquad \det g=1+O_{\scnorm}(r^N).
\end{gathered}
\end{equation}
Here $B$ is the Taylor polynomial of $\log g$ through degree $N-1$. The notation $O_{\scnorm}(r^d)$ includes the angular and $r\partial_r$ derivatives, with their scaled H\"older bounds, through a fixed finite order sufficient for the subsequent estimates. We retain $E_N$ in both the equation and the action.

The finite normalization also depends smoothly on the metric jets. At homogeneous degree $d\ge2$, its polynomial equation, in our Laplacian convention, has operator $r^2\Delta+(n-2)d(d-1)$. On $r^{2j}$ times a harmonic polynomial of degree $d-2j$, the first term has eigenvalue $2j(n-2+2d-2j)\ge0$. The second term is strictly positive. Each step is therefore a fixed finite-dimensional inverse applied to finitely many smooth metric jets. The inverses are uniform on compact sets of the finite parameters. Solving the parameter-dependent geodesic equation gives the normal coordinates, and the finite Taylor remainder gives \eqref{eq:finite-normal}.

Let $q=\max\{2,\lfloor a\rfloor+1\}>a$. At each $p_i$, exactly one of the following alternatives holds:

\begin{enumerate}
\item the first nonzero homogeneous summand $H_{d_i}$ of $B$ has $2\le d_i<a$;
\item $a\ge2$ is an integer and the first such summand has $d_i=a$;
\item all summands of degree at most $a$ vanish, and $g-\delta=O_{\scnorm}(r^q)$.
\end{enumerate}

Write $\mathcal C_<$, $\mathcal C_=$ and $\mathcal J$ for these sets of indices, and $\mathcal C=\mathcal C_<\cup\mathcal C_=$. In dimensions three, four and five, $\mathcal J=\{1,\ldots,k\}$.

The Yamabe inequality and compactness of $M$ show that the quadratic form of $P_g$ controls the $L^2$ norm. Since $R_g$ is bounded, it also controls the gradient norm. Thus $P_g$ is coercive on $H^1(M)$ and invertible. If $\varphi>0$ is its first eigenfunction, the transformation $v=\varphi w$ gives
$$
P_g(\varphi w)=\varphi\left[-\Delta_gw-2\langle\nabla\log\varphi,\nabla w\rangle+\lambda_1(P_g)w\right].
$$
The maximum principle for the operator in brackets proves positivity of the inverse. A local fundamental solution and the maximum principle away from its pole give a strictly positive Green function satisfying
$$
C^{-1}\dist_g(x,y)^{-m}\le G(x,y)\le C\dist_g(x,y)^{-m}.
$$
Local scaled Schauder estimates give the corresponding derivatives away from the diagonal. The same estimates hold uniformly for the auxiliary metrics obtained by flattening $g$ in smaller normal balls while leaving it unchanged outside balls of radius $2r_0$. Their quadratic forms differ from that of $P_g$ by $O(r_0^2)\|v\|_{H^1}^2$, so coercivity persists for small fixed $r_0$.

\begin{lemma}\label{lem:mass}
Suppose that $i\in\mathcal J$. Write $G_i=r^{-m}+h_i$ in the fixed normal coordinates. There exists $\mathfrak a_i\in\R$ such that
\begin{equation}\label{eq:mean-mass}
\overline{h_i}(r)\longrightarrow\mathfrak a_i,\qquad
r\partial_r\overline{h_i}(r)\longrightarrow0.
\end{equation}
The Green function metric $G_i^{4/m}g$ is complete, scalar flat and asymptotically flat, and its ADM mass is $2\mathfrak a_i$. If $(M,[g])$ is not conformally equivalent to the round sphere, then $\mathfrak a_i>0$.
\end{lemma}

\begin{proof}
Let $L_q(r)=1+|\log r|$ when $q=m$, and $L_q(r)=1$ otherwise. Applying the positive inverse to the error of a cut-off $r^{-m}$ gives
$$
|h_i|_{\scnorm,2,\alpha}\le C\{r^{q-m}+L_q(r)\}.
$$
For $|x|=r$, split the Newton integral into $|y|<r/2$, $|x-y|<r/2$, the remaining region with $|y|\le2r$, and $|y|>2r$. Each of the first three contributions is $O(r^{q-m})$. The last is bounded by a constant times $\int_r^{r_*}s^{q-m-1}\,ds$, which gives the power, logarithmic and bounded alternatives above.

If a symmetric tensor satisfies the Gauss and trace-free conditions, its linearized scalar curvature is $r^{-2}\diver_S\diver_SB_T$, including when the tangential tensor depends on $r$. Its spherical average is zero. The nonlinear scalar curvature average is therefore $O(r^{2q-2}+r^{N-2})$. Moreover, the exact radial identity and the determinant estimate in \eqref{eq:finite-normal} give
$$
\Delta_g r^{-m}-\Delta_\delta r^{-m}=O_{\scnorm}(r^{N-m-2}).
$$
Taking spherical averages of $P_gG_i=0$ yields
$$
(r^{m+1}\overline h_i')'=O(r^{2q-1}+r^{m+q-1}L_q(r)+r^{N-1}).
$$
The preceding pointwise estimate gives $r^{m+1}\overline h_i'\to0$. Integrating twice, using $2q>m$, gives
\begin{equation}\label{eq:mass-rate}
\overline h_i=\mathfrak a_i+O(r^{2q-m}+r^qL_q+r^{N-m}),\quad
r\overline h_i'=O(r^{2q-m}+r^qL_q+r^{N-m}).
\end{equation}
This proves the spherical-average limits in \eqref{eq:mean-mass}.

Under inversion $x=y/|y|^2$, $R=|y|=r^{-1}$, write $f=r^mG_i$ and $\varphi=f^{4/m}$. The inverted metric is $\varphi\bar g$, where $\bar g(y)=g(x)$ and $\bar g\,y=y$. Set $T=\tr\bar g-n$. Trace freeness gives $T,RT_R=O(R^{-2q}+R^{-N})$. Differentiating $y_i\bar g_{ij}=y_j$ shows that the ADM integrand is
$$
\nu_i\{\partial_j(\varphi\bar g_{ij})-\partial_i(\varphi\bar g_{jj})\}
=-(n-1)\varphi_R-T\varphi_R-\varphi(T_R+T/R).
$$
The last two terms integrate to zero in the limit. Formula \eqref{eq:mass-rate} gives limiting flux $4(n-1)\omega\mathfrak a_i$. Dividing by $2(n-1)\omega$ proves the mass formula.

For any $a<\beta<\min(q,m)$, the metric and its derivatives through order two differ from the Euclidean metric by $O_{\scnorm}(R^{-\beta})$. Its scalar curvature is zero. Since $G_i^{2/m}\asymp r^{-2}$, the end is complete; its complement is compact. The smooth case of \cite[Theorem 1.1]{KWW}, with empty singular set, therefore applies: the manifold is complete, its decay exponent exceeds $(n-2)/2$, and its scalar curvature is integrable. Its mass is nonnegative, and zero mass forces the Green function metric to be Euclidean.

If the mass is zero, let $y$ be the Euclidean isometry on $M\setminus\{p_i\}$. Then $|y|\asymp r^{-1}$, and
$$
v=\left(\frac2{1+|y|^2}\right)^aG_i
$$
is bounded above and below by positive constants and solves $P_gv=Qv^p$ away from $p_i$. Interior estimates give $|\nabla v|\le C/r$; its flux over a small sphere is $O(r^{n-2})$, which tends to zero. Also $v\in H^1$, since $\int_0^{r_*}r^{n-3}\,dr<\infty$. The vanishing flux extends the equation weakly across $p_i$. Since $v$ is bounded above and below by positive constants, local $W^{2,s}$ estimates for every finite $s$, followed by elliptic bootstrapping, give a smooth positive extension. The resulting metric agrees with the round metric under stereographic projection off $p_i$. Its metric completion adds one point, and the isometry extends smoothly across that point, for example by using distances squared to a set of nearby points as local coordinates. Thus $(M,[g])$ is the round conformal sphere. In the other case the mass is strictly positive.
\end{proof}
\section{The radial equation and its low spherical modes}\label{sect3}

For $t=-\log r$, a radial Euclidean solution has the form $r^{-a}v(t)$, where
$$
v''-a^2v+Qv^p=0,\qquad
(v')^2-a^2v^2+a^2v^{p+1}=-a^2\eps^2e.
$$
We choose the maximum at $t=0$. For $e$ in a compact positive interval and small $\eps>0$, the phase portrait gives a positive periodic solution $v_{\eps,e}$. Its minimum $\eta$ satisfies $\eta^2(1-\eta^{2/a})=\eps^2e$. The parameter $\eps$ is an energy parameter; it is not the exact minimum. As $\eps\to0$, the solutions on fixed intervals converge to $v_0(t)=\sech^a t$.

The energy identity gives $0<v<1$ and $|v'|<av$. Differentiating the energy with respect to $e$ gives
\begin{equation}\label{eq:wronskian}
v'\partial_ev'-v''\partial_ev=-\frac{a^2\eps^2}{2}.
\end{equation}
 Let $T_\eps(e)$ denote the least positive period of $v_{\eps,e}$.
The period is $(2/a)\log(1/\eps)+O(1)$ and the integral of $v^{2/a}$ over one period is bounded uniformly. The following estimates control both a single period and the accumulation of parameter derivatives over repeated periods.

\begin{lemma}\label{lem:delaunay-parameters}
For $e$ in a fixed positive compact interval, every fixed time derivative order $\ell$, and $0\le j\le3$,
$$
|\partial_e^j\partial_t^\ell v_{\eps,e}(t)|\le C(1+|t|)^jv_{\eps,e}(t),\qquad
|\partial_e^jT_\eps(e)|\le C_j\quad(1\le j\le3).
$$
The corresponding positive function $av-v'$ has relative derivatives bounded by a time polynomial as well. On fixed intervals, $v_{\eps,e}-v_0=O(\eps^2)$ with the time derivatives and two $e$ derivatives used here.
\end{lemma}

\begin{proof}
Put $z=\log\eta$. Logarithmic differentiation of $\eta^2(1-\eta^{2/a})=\eps^2e$ has coefficient $2-\frac2a\eta^{2/a}/(1-\eta^{2/a})$, bounded above and below by positive constants. Thus $\eta\asymp\eps$ and all fixed positive-order $e$ derivatives of $z$ are bounded. It therefore suffices to estimate derivatives with respect to $z$.

Fix a small level $\delta_0>0$. In the time from the minimum to this level use $v=\eta\cosh x$. The time integral is
$$
\frac1a\int_0^{\operatorname{arcosh}(\delta_0/\eta)}
\left[1-\eta^{2/a}\frac{\cosh^{2+2/a}x-1}{\sinh^2x}\right]^{-1/2}dx.
$$
The quotient extends smoothly and evenly at zero and its fixed $x$ derivatives are bounded by $C\cosh^{2/a}x$. Choose $\delta_0$ so that the bracket stays above $1/2$. Every positive $z$ derivative of its integrand is bounded by $C\eta^{2/a}\cosh^{2/a}x$, whose integral is bounded. Derivatives of the upper limit and their boundary terms are bounded too. The integral of the constant part gives $a^{-1}\log(1/\eta)+O(1)$ and bounded positive-order derivatives. Between $\delta_0$ and the maximum, substitute $v_{\max}-v=y^2$ near the upper turning point; the derivative of $v^2-v^{2+2/a}$ there stays away from zero. The resulting integral and its derivatives are uniformly bounded. This proves the period assertions and the bound on the integral of $v^{2/a}$.

For the solution on one period, the peak region $v\ge\delta_0/2$ has bounded time length and regular positive initial data. Its variation equations give bounded derivatives. In the intermediate region $K\eta\le v\le\delta_0$, with fixed large $K$, let $\tau_{\delta_0}\in(0,T_\eps(e)/2)$ be the unique time at which the decreasing trajectory reaches $\delta_0$. Then the implicit time equation is
$$
t=\tau_{\delta_0}+\frac1a\int_v^{\delta_0}
\frac{ds}{s\sqrt{1-s^{2/a}-\eta^2(1-\eta^{2/a})/s^2}}.
$$
The denominator is uniformly separated from zero. A positive $z$ derivative of the integrand is bounded by $C\eta^2/s^2$ times $ds/s$, which has bounded integral. The derivative of time with respect to $\log v$ has magnitude bounded above and below. Repeated implicit differentiation therefore bounds the fixed-time derivatives of $\log v$. In the remaining neck region write $v(t)=\eta V(t-T/2;\eta^{2/a})$. The equation for $V$ has initial data $(1,0)$ and smooth bounded coefficients on a fixed time interval; the period derivative estimates give the required relative bounds there. Reflecting gives the other half-period.

Locally write an arbitrary time as $t=\bar t+kT(e)$ with fixed integer $k$. Each $e$ derivative introduces at most one factor $k$; the one-period bounds and bounded derivatives of $T$ give the stated polynomial. Overlapping period representations give the same bound at their boundaries, and $|k|\le C(1+|t|)$. For $J=av-v'$, when $v'\le0$ use $J\ge av$ and the solution bounds directly. When $v'\ge0$, use
$$
J=\frac{a^2(v^{p+1}+\eps^2e)}{av+v'}
$$
and $av+v'\ge av$. Logarithmic derivatives of its positive numerator have the same polynomial bounds. Finally the peak is a smooth function of $\eps^2e$, and finite-interval ODE dependence proves the last assertion.
\end{proof}

To solve the radial and first spherical projections together, we allow the special conformal parameter to vary along the cylinder. For $b\in\R^n$, let
$$
T_b(y)=\frac{y-|y|^2b}{1-2b\cdot y+|b|^2|y|^2},\qquad
\Omega_b(y)=(1-2b\cdot y+|b|^2|y|^2)^{-1}.
$$
If $b=b(t)$ and $\mathbf q=rb'$, the metric $r^{-2}\Omega_b^{-2}T_b^*\delta$ is exactly
\begin{equation}\label{eq:moving-cylinder}
\gamma_{\mathbf q}=(1-\mathbf q\cdot\theta)^2dt^2+
|d\theta-\nabla_S(\mathbf q\cdot\theta)dt|^2.
\end{equation}
Every occurrence of $b$ in this formula is evaluated at the original $t=-\log|y|$.

For a radial function $w$, put $J=aw-w'$ and $E_w=(w')^2-a^2w^2+a^2w^{p+1}$. Let $\langle\cdot\rangle$ denote spherical average, and define the vector function
$$
\mathcal B(\mathbf q)=n\left\langle\frac{\theta}{(1-\mathbf q\cdot\theta)^2}\right\rangle.
$$
Thus $D\mathcal B(0)=2I$. For $F=w''-a^2w+Qw^p$, substitution in the conformal Laplacian of \eqref{eq:moving-cylinder} gives the full radial residual
\begin{align*}
\mathcal N_{\gamma_{\mathbf q}}w
={}&F-\frac{(\mathbf q\cdot\theta)(2-\mathbf q\cdot\theta)}{(1-\mathbf q\cdot\theta)^2}J'
+\frac{(\mathbf q\cdot\theta)(1+a\mathbf q\cdot\theta)}{(1-\mathbf q\cdot\theta)^2}J\\
&-\frac{\mathbf q'\cdot\theta+|\mathbf q|^2-(\mathbf q\cdot\theta)^2}{(1-\mathbf q\cdot\theta)^3}J.
\end{align*}
Represent the first projection by $n\langle f\theta\rangle$. The $J'$ coefficient projects to $-\mathcal B$, and the $\mathbf q'$ coefficient to $-D\mathcal B/2$. The remaining $J$ coefficient projects to $\mathcal B/2$. For the latter identity rotate $\mathbf q=\kappa e_1$; after subtraction and a common denominator its difference is $-\kappa^2/2$ times
$$
\int_{-1}^{1}\frac d{dx}\left[\frac{x^2(1-x^2)^{(n-1)/2}}{(1-\kappa x)^2}\right]dx=0.
$$
The boundary values vanish because $n\ge3$. These three coefficients give the exact flux identity
\begin{equation}\label{eq:once-flux}
\Pi_1\mathcal N_{\gamma_{\mathbf q}}w
=-\frac1{rJ}\frac d{dt}\left(\frac{rJ^2}{2}\mathcal B(\mathbf q)\right).
\end{equation}
The calculation applies to the forced trajectory $w$, so \eqref{eq:once-flux} remains valid in the coupled source equations.

The mean projection involves the scalar functions
$$
\mathcal U=\langle(1-\mathbf q\cdot\theta)^{-1}\rangle-1,
\qquad \mathcal A=\langle(1-\mathbf q\cdot\theta)^{-2}\rangle-1,
$$
and the combination $\mathcal D=D\mathcal A(D\mathcal B)^{-1}\mathcal B-\mathcal A$. These functions are defined on a fixed small ball: $\mathcal U,\mathcal A,\mathcal D$ vanish to order two and $D\mathcal A(D\mathcal B)^{-1}$ vanishes to order one. If the desired first projection is $f_1$, and $F=w''-a^2w+Qw^p$, the mean identity is
\begin{equation}\label{eq:mean-eliminate}
\Pi_0\mathcal N_{\gamma_{\mathbf q}}w
=(1-\mathcal D)F+Qw^p\mathcal D
-\frac{n-1}{2}J(\mathcal D+\mathcal U)+D\mathcal A(D\mathcal B)^{-1}f_1.
\end{equation}
Before elimination, the nonradial remainder has average
$$
-\frac J2\{D\mathcal A\,\mathbf q'+(2J'/J-1)\mathcal A+(n-1)\mathcal U\}.
$$
Substituting the derivative of \eqref{eq:once-flux}, and then $J'=-aJ+Qw^p-F$, gives \eqref{eq:mean-eliminate}. Only $F$ occurs in this identity, so $C^{0,\alpha}$ source regularity suffices.

\begin{lemma}\label{lem:low}
Fix $\beta_0>0$, $\beta_1>1$ and a bounded ball in the uniform unit-interval $C^{0,\alpha}$ norm for sources $F_0(t)\in\R$ and $F_1(t)\in\R^n$. Prescribe $e_\infty$ in a compact positive interval and bounded $\phi_\infty,b_\infty$. For a sufficiently large fixed $t_0$, independent of small $\eps$, there is a unique small solution of
$$
\Pi_0\mathcal N_{\gamma_{\mathbf q}}w=\eps^2w^{-1}r^{\beta_0}F_0,
\qquad \Pi_1\mathcal N_{\gamma_{\mathbf q}}w=\eps^2w^{-1}r^{\beta_1}F_1,
$$
with terminal energy $e_\infty$, phase $\phi_\infty$ and center $b_\infty$. It satisfies $-E_w\asymp\eps^2$, $w\ge C^{-1}\eps$, and
\begin{equation}\label{eq:low-estimates}
|\mathbf q|+|\mathbf q'|\le Cr^{\beta_1},\quad
wJ(|\mathbf q|+|\mathbf q'|)\le C\eps^2r^{\beta_1},\quad
|b-b_\infty|\le Cr^{\beta_1-1}.
\end{equation}
The solution is twice Fr\'echet differentiable in the sources in their stated norms. When the terminal parameters vary, two derivatives exist in any fixed strictly weaker positive radial weight and any strictly weaker first-mode weight greater than one.
\end{lemma}

\begin{proof}
Write the actual Cauchy state as $(w,w')=(v_{\eps,e(t)}(t+\phi(t)),v_s(t+\phi(t)))$. On candidate paths the two entries are independent state coordinates. The variation equations below impose the compatibility that the derivative of the first equals the second. Formula \eqref{eq:mean-eliminate} determines
$$
F=\frac{f_0-Qw^p\mathcal D+(n-1)J(\mathcal D+\mathcal U)/2-D\mathcal A(D\mathcal B)^{-1}f_1}{1-\mathcal D}.
$$
The radial variation equations, obtained by inverting \eqref{eq:wronskian}, are
$$
e'=-\frac{2v_sF}{a^2\eps^2},\qquad
\phi'=\frac{2v_eF}{a^2\eps^2}.
$$
They imply $v_ee'+v_s\phi'=0$, and hence give the required Cauchy compatibility at a fixed point.

The terminal flux condition in \eqref{eq:once-flux} gives
$$
\frac{rJ^2}{2}\mathcal B(\mathbf q)
=\eps^2\int_t^\infty r(s)^{\beta_1+1}\frac{J(s)}{w(s)}F_1(s)\,ds,
\qquad b=b_\infty-\int_t^\infty\mathbf q(s)/r(s)\,ds.
$$
For negative energy, $J(aw+w')=a^2w^{p+1}-E_w$ and $aw+w'\le2aw$ give $wJ\ge C^{-1}\eps^2$. Also
$$
\left(\log\frac Jw\right)'\le2+Cr^{\min(\beta_0,\beta_1)},\qquad
\frac{J(t+s)}{w(t+s)}\le C e^{2s}\frac{J(t)}{w(t)}.
$$
Consequently the integral defining $\mathcal B(\mathbf q)$ is bounded by $C\eps^2r^{\beta_1+1}J/w$, and inversion of $\mathcal B$ gives \eqref{eq:low-estimates}.

We construct the fixed point in these state coordinates, imposing compatibility only through the integral equations. Choose $1<\ell<\beta_1$ and $0<\rho<\min(\beta_0,2\ell)$. Use $e-e_\infty,\phi-\phi_\infty$ in the norm $\sup r^{-\rho}|\cdot|$, and write the flux as $\eps^2r^{\ell+1}(J_*/v_*)P(t)$ with bounded $P$, where $v_*=v_{\eps,e_\infty}(t+\phi_\infty)$ and $J_*=av_*-v_*'$. Define $\mathbf q$ by the displayed inverse of $\mathcal B$, then define $b$ by its future integral. The positive factor $\eps^2/(wJ)$ is bounded before differentiation. The two integrations of the radial equations and the flux equation define the fixed point map.

Its four derivative blocks are bounded by quantities of the form
$$
C\sup_{t\ge t_0}(1+t)^Lr^{\kappa},\qquad \kappa>0.
$$
More explicitly, the radial-to-radial block has exponent $\min(\beta_0,2\ell)$, the flux-to-radial block has exponent $\min(\beta_0+\ell-1,2\ell)-\rho$, the radial-to-flux block has exponent $\rho$, and the flux-to-flux block has exponent $\ell-1$. All are positive. Choose $t_0$ to make their sum less than $1/2$ on a slightly larger ball than the one used. The map is a twice differentiable contraction on this open ball. At its fixed point the integral equations give the original mean and flux equations and hence Cauchy compatibility.

For source derivatives the reference $v_*$ is fixed. Differentiating the flux integral gives the original $\beta_1$ weight, since each relative derivative of $J/w$ is bounded by a polynomial times a decaying state variation. The identities for $w'/w$, $J'/J$ and the flux then give the unit-interval H\"older estimates without differentiating $F_0$ or $F_1$. Differentiation in the terminal parameters introduces time polynomials without a decaying state variation. A strictly weaker weight absorbs these factors. Difference quotients in a finite decreasing sequence of weights, followed by the differentiated contraction equations, give the first two Fr\'echet derivatives and their continuity. Positivity and the negative energy band follow by keeping $e$ in the prescribed positive interval.
\end{proof}

\begin{lemma}\label{lem:fixed-entry}
Fix a terminal energy $e_\infty$ in a positive compact interval, a bounded open source ball as in Lemma~\ref{lem:low}, and a compact set of entrance parameters $(\phi_0,b_0)$ strictly inside a larger parameter box. For all sufficiently small $\eps$, there is a unique nearby choice of terminal phase and center, depending jointly $C^2$ on the source and the entrance parameters, such that
\begin{equation}\label{eq:fixed-entry}
\phi(t_0)=\phi_0,\qquad b(t_0)=b_0.
\end{equation}
The terminal phase and center equal $(\phi_0,b_0)+o_{C^2}(1)$ on every smaller source ball and compact parameter set. For every fixed $T>t_0$, the resulting trajectory satisfies
\begin{equation}\label{eq:entry-cauchy}
\|w-v_0(\cdot+\phi_0)\|_{C^2 C^{2,\alpha}([t_0,T])}
+\|b-b_0\|_{C^2 C^{2,\alpha}([t_0,T])}
+\|\mathbf q\|_{C^2 C^{1,\alpha}([t_0,T])}\le C_T\eps^2.
\end{equation}
Here $C^2$ includes joint source and entrance derivatives, in operator norm, with $\eps$ and $e_\infty$ fixed; the bounds are uniform for $e_\infty$ in the prescribed interval. In particular the static bubble and its static coordinate map are independent of the source.
\end{lemma}

\begin{proof}
Write $I=(F_0,F_1)$ and first use the terminal parametrization of Lemma~\ref{lem:low}. With terminal energy fixed, a derivative of the reference $v_{\eps,e_\infty}(t+\phi_\infty)$ in $\phi_\infty$ is a time derivative, rather than an energy derivative. The equation and the positive identity
$$
(av-v')(av+v')=a^2(v^{p+1}+\eps^2e_\infty)
$$
give uniform bounds for the required time derivatives of $\log v$ and $\log(av-v')$. To check the latter near a rising neck, divide the positive numerator by $av+v'$; the quotient $v^p/(av-v')$ is bounded, so differentiating $J'=-aJ+Qv^p$ introduces no small denominator. The differentiated contraction equations of Lemma~\ref{lem:low} therefore give joint source, phase and center derivatives. Energy and phase increments retain a positive time weight; their time-polynomial factors are absorbed by a slightly smaller positive weight.

With the terminal energy fixed, the original first-mode weight is preserved. Each of the first two relative derivatives of $J/w$ is bounded on unit cylinders: reference phase derivatives are bounded by the preceding identities, whereas derivatives of $e(t)-e_\infty$ and $\phi(t)-\phi_\infty$ have positive decay and absorb the energy-derivative polynomials of Lemma~\ref{lem:delaunay-parameters}. Differentiating the positive flux integral thus preserves its majorant
$$
\frac{J(t+s)}{w(t+s)}\le Ce^{2s}\frac{J(t)}{w(t)}.
$$
Its product with $r^{\beta_1+1}$ is integrable because $\beta_1>1$. This proves the joint $C^2$ bounds for $\mathbf q$ in its original first-mode weight and for $b-b_\infty$ in weight $\beta_1-1$, without differentiating the sources in time.

Define the actual entrance map by
$$
\mathcal G_\eps(I,\phi_\infty,b_\infty)
=(\phi(t_0),b(t_0)).
$$
Its difference from $(\phi_\infty,b_\infty)$ is the pair of integrals $-\int_{t_0}^\infty\phi'$ and $-\int_{t_0}^\infty\mathbf q/r$. The integrands and their joint derivatives through order two have a common bound $C(1+t)^L e^{-\eta t}$ for some $\eta>0$. On every fixed cylinder, $w$ and $J$ have positive lower bounds. The flux identity gives $\mathbf q=O_T(\eps^2)$ there, the mean equation gives $F=O_T(\eps^2)$, and $v_e=O_T(\eps^2)$ gives $\phi'=O_T(\eps^2)$. Differentiating these identities gives the same estimates through joint order two. Split each integral at $T$. The fixed-cylinder part tends to zero with $\eps$, and the common integrable bound makes the remaining tail uniformly small as $T\to\infty$. Thus
$$
\mathcal G_\eps(I,\phi_\infty,b_\infty)
=(\phi_\infty,b_\infty)+o_{C^2}(1).
$$

For prescribed $(\phi_0,b_0)$, iterate
$$
(\phi_\infty,b_\infty)
=(\phi_0,b_0)-\{\mathcal G_\eps(I,\phi_\infty,b_\infty)
-(\phi_\infty,b_\infty)\}.
$$
On a fixed small ball about $(\phi_0,b_0)$ this is a self-map with derivative norm less than $1/2$, uniformly on the smaller source ball and compact entrance set. It gives the unique nearby terminal parameters. Implicit differentiation gives their joint $C^2$ dependence; their source derivatives through order two are $o(1)$, and their entrance differential is the identity plus $o(1)$. This proves \eqref{eq:fixed-entry}.

After this substitution the source derivatives of $e$ remain bounded on fixed cylinders. Since $\phi(t_0)=\phi_0$, integration of $\phi'=O_T(\eps^2)$ gives $\phi-\phi_0=O_T(\eps^2)$. The finite-interval expansion in Lemma~\ref{lem:delaunay-parameters} now gives $w-v_0(t+\phi_0)=O_T(\eps^2)$ with joint derivatives. The equation for $w''$ gives its $C^{2,\alpha}$ estimate using only the $C^{0,\alpha}$ source. Integrating $b'=\mathbf q/r$ from its fixed entrance value, and differentiating the flux equation once in time, proves the remaining bounds in \eqref{eq:entry-cauchy}. The argument uses bounded entrance energy and the fixed phase and center in \eqref{eq:fixed-entry}.
\end{proof}

\section{High modes and the nonlinear singular end}\label{sec:ends}

The spherical eigenvalue of degree $j$ is $j(j+2a)$. For an actual forced radial trajectory from Lemma~\ref{lem:low}, let
$$
\mathcal L_w=\partial_t^2+\Delta_S-a^2+Qp\,w^{p-1}.
$$
The high-mode estimates use two normalizations: one for the geometric first response and one for the nonlinear remainder.

\begin{lemma}\label{lem:high}
Fix a weight $\delta\ge0$.
Suppose $w''-a^2w+Qw^p=F$, $E_w=(w')^2-a^2w^2+a^2w^{p+1}<0$, and $F/w$ has a uniform unit-interval $C^{0,\alpha}$ bound and sufficiently small supremum. On spherical modes $j\ge j_0\ge2$, the zero entrance solution of $-\mathcal L_wh=f$ satisfies
$$
\|h/w\|_{C^{2,\alpha}_\delta}\le C\|f/w\|_{C^{0,\alpha}_\delta}
$$
if
$$
j_0(j_0+2a)-2(a+1)-2a\delta-\delta^2-\|F/w\|_\infty>0.
$$
It also satisfies
$$
\|wh/\eps^2\|_{C^{2,\alpha}_\delta}\le C\|wf/\eps^2\|_{C^{0,\alpha}_\delta}
$$
if
$$
j_0(j_0+2a)-(4a+2)-\delta^2\left(1+\frac{a^2}{4a+2}\right)-\|F/w\|_\infty>0.
$$
Here $\|z\|_{C^{k,\alpha}_\delta}=\sup_{T\ge t_0}e^{\delta T}\|z\|_{C^{k,\alpha}([T,T+1]\times\Sph^{n-1})}$, including the boundary norms at $t_0$. The inverse is taken in this weighted space with zero entrance value. The constants may depend on the fixed weight, the positive gap in the displayed inequality and the uniform coefficient bounds, but are independent of $\inf w$ and $\eps$.
\end{lemma}

\begin{proof}
For $h=w\xi$, direct differentiation gives
$$
-w^{-1}\mathcal L_w(w\xi)=-\xi''-2(w'/w)\xi'-\Delta_S\xi
-\{2(a+1)w^{2/a}+F/w\}\xi.
$$
Conjugation by $e^{-\delta t}$ contributes $2\delta(w'/w)-\delta^2$ to the zeroth-order coefficient. Since $\delta\ge0$ and $|w'/w|\le a$,
$$
2\delta\frac{w'}w-\delta^2\ge-2a\delta-\delta^2.
$$
Together with $w\le1$, this gives the first lower bound in the statement.

For the second normalization, put $h=\eps^2w^{-1}\xi$, $z=w'/w$ and $d=-E_w/w^2\ge0$. The conjugated expression before introducing the exponential weight is
$$
-w\mathcal L_w(w^{-1}\xi)
=-\xi''+2z\xi'-\Delta_S\xi+
\{2d-(4a+2)w^{2/a}+F/w\}\xi.
$$
After the exponential conjugation, its spherical lower bound is
$$
j_0(j_0+2a)+2d-(4a+2)w^{2/a}-2\delta z-\delta^2-\|F/w\|_\infty.
$$
Since $z^2=a^2(1-w^{2/a})-d$, completing the square in $z$ gives the second displayed bound.

To construct the inverse, first impose zero Dirichlet data at both ends of a finite cylinder. Before exponential conjugation, the two operators are symmetric with measures $w^2dt\,d\theta$ and $w^{-2}dt\,d\theta$, respectively; their positive spherical forms give existence there. For the $L^2(S)$ norm of the conjugated solution, the Hilbert space Kato inequality gives a one-dimensional differential inequality with bounded drift and a strictly positive zeroth order coefficient. The maximum principle bounds this norm by the corresponding source norm independently of the cylinder length. Unit-cylinder elliptic estimates and Schauder estimates then give the asserted bounds. The coefficients are uniformly H\"older because $z'=a^2-Qw^{2/a}+F/w-z^2$ is controlled. A diagonal limit gives the half-cylinder solution. The same one-dimensional inequality, compared with a decaying exponential barrier from a distant terminal boundary, proves uniqueness. This argument also proves all estimates for source differences.
\end{proof}

For the relative estimate one may use even weight $1$ for $j\ge2$ and odd weight $2$ for $j\ge3$. For the energy normalization one may use even weight $1/\sqrt{a+1}$ and odd weight $2$. We fix henceforth $\delta=1/(2\sqrt{a+1})$ for the weaker even weight. The distinction between the two normalizations is essential: the first geometric response is relative to $w$, while the nonlinear remainder is measured relative to $\eps^2/w$.

We next express the geometric error in the moving coordinates. Apply the finite normalization of Section~2, now smoothly in the static parameter $b$, to $\Omega_b^{-2}T_b^*g_\tau$, where $g_\tau-\delta=O_{\scnorm}(\tau r^2)$ on a fixed ball. Let $N_{\tau,b}$ be its normal coordinate map and $f_{\tau,b}$ its conformal logarithm. Then
$$
K_{\tau,b}=T_b\circ N_{\tau,b},\qquad
A_{\tau,b}=(\Omega_b\circ N_{\tau,b})e^{-f_{\tau,b}},\qquad
\bar g_{\tau,b}=A_{\tau,b}^{-2}K_{\tau,b}^*g_\tau=e^{B_{\tau,b}}+E_{\tau,b}.
$$
All the quantities are expressed in the final normal variable. They satisfy
$$
(N_{\tau,b}-y)/r,\ f_{\tau,b},\ \bar g_{\tau,b}-\delta=O_{\scnorm}(\tau r^2),
\quad B_{\tau,b}y=0,\quad\tr B_{\tau,b}=0,\quad E_{\tau,b}=O_{\scnorm}(\tau r^N).
$$
These estimates hold after the required finite number of $b$ derivatives by the finite polynomial inverses and the parameter-dependent geodesic equation. Throughout this construction, $B_{\tau,b}$ denotes the full finite Taylor field, with its nonlinear dependence on $\tau$.

For $b=b(t(y))$, set $K(y)=K_{\tau,b(t(y))}(y)$, $A(y)=A_{\tau,b(t(y))}(y)$ and $\Gamma=r^{-2}A^{-2}K^*g_\tau$. The differential factors as
$$
DK=D_yK_{\tau,b}\left(I+\sum_lX_{\tau,b,l}\otimes db_l\right),\qquad
X_{\tau,b,l}=(DN_{\tau,b})^{-1}\{\partial_{b_l}N_{\tau,b}+2N_lN-|N|^2e_l\}.
$$
In particular $X_{\tau,b,l}-(2y_ly-r^2e_l)=O_{\scnorm}(\tau r^3)$. The map is a diffeomorphism fixing the puncture on a sufficiently small ball, since its differential is uniformly close to the identity.

For any cylinder function $Z$, the exact operator difference has the form
\begin{equation}\label{eq:geometric-split}
(\Delta_\Gamma-cR_\Gamma)Z
=(\Delta_{\gamma_{\mathbf q}}-cR_{\gamma_{\mathbf q}})Z
+\mathscr S_bZ+\mathscr R_bZ+\mathscr D_bZ.
\end{equation}
Here $\mathscr S_b$ is the full linear metric variation at the Euclidean metric with tensor $B_{\tau,b}+E_{\tau,b}$, transported with the cylinder powers. Its Gauss part is
$$
-B_T:\nabla_S^2Z-(\diver_SB_T)\cdot\nabla_SZ-c(\diver_S\diver_SB_T)Z.
$$
Define the reference metric
$$
\Gamma_{\rm ref}(y)=r^{-2}\bigl(e^{B_{\tau,b(t(y))}(y)}+E_{\tau,b(t(y))}(y)\bigr),
$$
with the path dependence included in all derivatives. The remainders are
\begin{align*}
\mathscr R_b={}&(\Delta_{\Gamma_{\rm ref}}-cR_{\Gamma_{\rm ref}})
-(\partial_t^2+\Delta_S-a^2)-\mathscr S_b,\\
\mathscr D_b={}&(\Delta_\Gamma-cR_\Gamma)
-(\Delta_{\gamma_{\mathbf q}}-cR_{\gamma_{\mathbf q}})\\
&-(\Delta_{\Gamma_{\rm ref}}-cR_{\Gamma_{\rm ref}})
+(\partial_t^2+\Delta_S-a^2).
\end{align*}
These definitions give \eqref{eq:geometric-split} exactly. In particular, $\mathscr D_b=0$ when $\tau=0$ or the path is constant. On unit cylinders,
\begin{align*}
|\mathscr S_bZ|&\le C\tau r^2|Z|_{2,\alpha},&
|\mathscr R_bZ|&\le C\tau^2r^4|Z|_{2,\alpha},\\
|\mathscr D_bZ|&\le C\tau\{|\mathbf q||Z|_{2,\alpha}+|\mathbf q'|\,|Z|_{1,\alpha}\}.&&
\end{align*}
The odd coefficient bounds gain one radial power in the first two terms. On a radial $w$, the Gauss part of $\mathscr S_bw$ has no zero or first spherical mode; its even and odd parts start at $r^4w$ and $r^5w$. The low projections contributed by $E_{\tau,b}$ are $O(\tau r^Nw)$ and are kept.

The scalar curvature terms also require only two derivatives of the coordinate maps. For $h=K^*g_\tau$,
$$
R_{A^{-2}h}=A^2\{R_{g_\tau}\circ K+2(n-1)\Delta_h\log A
-(n-1)(n-2)|d\log A|_h^2\}.
$$
The right side uses only $DK,D^2K,A,DA,D^2A$. Terms containing $b''$ are linear and act on at most first derivatives of $Z$. There is no $b'''$ term. Twice differentiating an $O_{\scnorm}(\tau r^2)$ static coefficient along $b(t)$ produces additional terms bounded by $C\tau\{r(|\mathbf q|+|\mathbf q'|)+|\mathbf q|^2\}$. This proves the mixed bound in \eqref{eq:geometric-split}, including its unit-cylinder H\"older and two parameter derivatives.

\begin{proposition}\label{prop:end}
Suppose $\tau\le C\eps$ and $\tau\eps^{-4/n}\to0$. Fix a terminal energy in a positive compact interval and a compact set of actual entrance parameters $(\phi_0,b_0)$ as in Lemma~\ref{lem:fixed-entry}. For a sufficiently large fixed $t_0$, the equation $\mathcal N_\Gamma U=0$ has a positive singular end
$$
U=w+V+H_w^+\sigma+h,
$$
where $w,b$ solve the common low-mode equations, $\mathcal L_wV=-\mathscr S_bw$, $h$ is high with zero entrance value, and $H_w^+$ is the high-mode homogeneous Poisson operator. The entrance data $\sigma$ may vary in a fixed ball of radius $\Sigma$, provided
\begin{equation}\label{eq:sigma-condition}
\eps^2+\tau^2=O(\Sigma),\qquad \Sigma=o(\eps).
\end{equation}
The end depends twice differentiably on the finite parameters in strictly weaker weights and has $U\ge w/2\ge C^{-1}\eps$. After the complete first response has been matched, its remaining low Cauchy discrepancy from the static bubble with entrance parameters $(\phi_0,b_0)$ is $O_{C^2}(\eps^2)$. For raw entrance variations,
\begin{equation}\label{eq:raw-entrance}
\|D_\sigma I\|\le C\eps^{-1},\quad
\|D_\sigma^2I\|\le C\eps^{-2},\quad
\|D_\sigma(\mathfrak d,\mathfrak n)\|\le C\eps,\quad
\|D_\sigma^2(\mathfrak d,\mathfrak n)\|\le C.
\end{equation}
Here $I$ denotes the normalized sources, and $(\mathfrak d,\mathfrak n)$ are the physical Cauchy discrepancies after matching the complete first response.
\end{proposition}

\begin{proof}
Use the linear source space
$$
C^{0,\alpha}_{\rm unif}(\R\oplus\R^n)
\oplus C^{0,\alpha}_{\delta,\mathrm{even},j\ge2}
\oplus C^{0,\alpha}_{2,\mathrm{odd},j\ge3}.
$$
Given $I=(F_0,F_1,G_e,G_o)$ in a fixed ball, apply Lemmas~\ref{lem:low} and \ref{lem:fixed-entry} with $\beta_0=\delta$, $\beta_1=2$. Thus the terminal phase and center depend on $I$, while $\phi(t_0)=\phi_0$ and $b(t_0)=b_0$ hold exactly. Solve
$$
\mathcal L_wV=-\mathscr S_bw,\qquad
-\mathcal L_wh=\eps^2w^{-1}(G_e+G_o),\qquad h(t_0)=0.
$$
Prescribe the geometric entrance value of the high part of $V$. For the low part due to $E_{\tau,b}$, use the fundamental matrix of each of the finitely many low equations and integrate its source from $t$ to infinity. Its norm is bounded by $Ce^{C_n|s-t|}$ because all coefficients are bounded; taking the finite normal order $N$ larger than $C_n+a+3$ makes this future integral and its two parameter derivatives converge. Choose $N$ also larger than $a+\sqrt{(a+1)(a+2)}+3$. This defines the strong future inverse with decay order $N$. Its entrance value is not set to zero. Lemma~\ref{lem:high} gives
$$
V_e=O(\tau rw),\quad V_o=O(\tau r^2w),\quad
h_e=O(\eps^2w^{-1}r^\delta),\quad h_o=O(\eps^2w^{-1}r^2).
$$
First take $\sigma=0$ and write $Z=V+h$. Subtracting the equations already solved gives the exact residual
$$
\mathcal N_\Gamma(w+Z)=\mathcal N_{\gamma_{\mathbf q}}w+\mathcal L_wh+\mathscr E,
$$
where $\mathscr E$ consists of the nonlinear power remainder, the flat moving operator acting on $Z$, $\mathscr S_bZ+\mathscr R_b(w+Z)$, and $\mathscr D_b(w+Z)$. Define the source update by
\begin{align*}
\widetilde F_0&=-r^{-\delta}(w/\eps^2)\Pi_0\mathscr E,&
\widetilde F_1&=-r^{-2}(w/\eps^2)\Pi_1\mathscr E,\\
\widetilde G_e+\widetilde G_o&=(w/\eps^2)\Pi_{\ge2}
\{\mathcal N_{\gamma_{\mathbf q}}w+\mathscr E\}.&&
\end{align*}
The fixed point therefore solves all spherical projections of the same equation.

We estimate this update in the stated source norm. In the positive relative ball $|Z|/w<1/2$, the exact power remainder is $Qw^p(Z/w)^2\Theta_p(Z/w)$, with $\Theta_p$ smooth on this interval. After multiplication by $w/\eps^2$, the coefficients of $V^2,Vh,h^2$ are bounded respectively by $C$, $C\tau$ and $C\eps^{\min(2,p-1)}$. Their lowest even radial powers are $2$, $1+\delta$, $2\delta$, and their lowest first or odd powers are $3$, $2+\delta$, $2+\delta$. These exceed the target powers $\delta$ and $2$.

The actual low flux factorization gives $\mathbf q,\mathbf q'=\eps^2(wJ)^{-1}$ times a controlled function of rate $2$. Thus the moving term acting on $V_e$ has normalized size at most $C\tau(w/J)r^3$. Optimizing
$$
J(aw+w')=a^2w^{p+1}-E_w,
\qquad aw+w'\le2aw
$$
gives $w/J\le C\eps^{-4/n}$: minimize $w^{2/a}+\eps^2/w^2$ over $w>0$. The moving geometric term on $w$ is controlled by $C\tau\eps^{-4/n}r^2$. Its first projection can be at the target power, so the displayed small factor is needed. On $h$, the normalized moving term has powers at least $2+\delta$. Static linear and quadratic metric terms gain at least two and four radial powers. The remaining high part of $\mathcal N_{\gamma_{\mathbf q}}w$ is $\eps^2w^{-1}$ times a function of even rate $4$ and odd rate $6$. Together these estimates control every term in $\mathscr E$ and the remaining high projection.

All estimates hold for two source derivatives. For the source derivatives, factor the fields as $V=wv$, $h=\eps^2w^{-1}\eta$ and $\mathbf q=\eps^2(wJ)^{-1}Z_q$ before differentiating. The coefficients of the power terms are exactly
$$
(\tau^2/\eps^2)w^{p+1},\qquad \tau w^{p-1},\qquad \eps^2w^{p-3}.
$$
Their relative derivatives have the same bounds. The other terms are smooth products of $v,\eta,w'/w,(w'/w)',Z_q,Z_q'$ and the finite metric jets. No source derivative in time is introduced. The odd part of the power remainder is the antisymmetric difference at the odd argument; its lowest terms contain both an even and an odd factor, or at least three odd factors. Thus the powers used above persist for nonintegral $p$.

On a source ball of radius $R_{\rm src}$ the update and its first two source derivatives are bounded by
$$
C_{R_{\rm src}}\{e^{-\kappa t_0}+\tau+\eps^{\min(2,p-1)}+\tau\eps^{-4/n}\},\qquad\kappa>0.
$$
The same bounds hold for two derivatives in the terminal phase and center at fixed terminal energy. Their reference derivatives are bounded time derivatives; the decaying energy increments absorb all energy-derivative polynomials, as in Lemma~\ref{lem:fixed-entry}. Differentiating the two high inverses uses bounded coefficient derivatives in their original positive weights. Hence the source update remains $C^2$ in these group parameters in the stated source space. Substitution of the terminal parameters supplied by Lemma~\ref{lem:fixed-entry} changes its source derivative by at most a bounded group derivative times $o(1)$. First choose $t_0$, then the small scales, so that the resulting contraction constant is less than $1/2$. The fixed point gives the end with $\sigma=0$ and the prescribed entrance phase and center.

For nonzero $\sigma$, use $J,w\ge C(t_0)^{-1}r^a$ and the stronger high-mode estimate proved at the start of Section~\ref{sec:inverse}. That estimate depends only on the forced radial trajectory, so its proof is independent of this nonlinear construction. It gives
$$
(H_w^+\sigma)_e=O(\Sigma r^{a+1}),\qquad
(H_w^+\sigma)_o=O(\Sigma r^{a+2}).
$$
Both $H_w^+\sigma/w$ and $H_w^+\sigma/J$ have corresponding bounds $C\Sigma r$, $C\Sigma r^2$. The extra normalized residual terms have small coefficient
$$
\Sigma+\frac{\tau\Sigma+\Sigma^2}{\eps^2}\longrightarrow0.
$$
For instance the power products $V(H_w^+\sigma)$ and $(H_w^+\sigma)^2$ have even rate $2$, odd rate $3$, with coefficients $\tau\Sigma/\eps^2$ and $\Sigma^2/\eps^2$; the product with $h$ has coefficient $\Sigma$. The moving term uses the quotient by $J$ and hence has coefficient $\Sigma$, with no extra negative power of $\eps$. The source update is therefore a contraction on each fixed normalized entrance ball once the small scales are chosen. Differentiating it in raw $\sigma$ gives the first two bounds in \eqref{eq:raw-entrance}.

For the physical Cauchy data put $\Omega=rA$, so that $u\circ K=\Omega^{-a}U$ and $K^*g=\Omega^2\Gamma$. With $\rho_\Gamma$ the volume density in $(t,\theta)$, its conormal, after removal of the common positive factor $\Omega^a$, is
\begin{equation}\label{eq:physical-conormal}
\rho_\Gamma\Gamma^{tj}(\partial_jU-aU\partial_j\log\Omega).
\end{equation}
Use the static map with parameter $b_0$. By \eqref{eq:fixed-entry}, the static core and moving end have identical boundary maps and tangential derivatives. Their conormal difference is $O(\mathbf q(t_0))=O_{C_I^2}(\eps^2)$ by \eqref{eq:entry-cauchy}. The entire first response, including its nonzero low component from $E_{\tau,b}$, is matched on both sides. The remaining low discrepancy is therefore $w-v_0(\cdot+\phi_0)$, and $h_t(t_0)$ is $O_{C_I^2}(\eps^2)$ by its energy-normalized bound at the fixed entrance. Consequently $\mathfrak d,\mathfrak n=O_{C_I^2}(\eps^2)$. The raw chain rule gives
$$
\|D_\sigma(\mathfrak d,\mathfrak n)\|
\le C\eps^2(1+\|D_\sigma I\|)\le C\eps,\qquad
\|D_\sigma^2(\mathfrak d,\mathfrak n)\|
\le C\eps^2(\|D_\sigma I\|^2+\|D_\sigma^2I\|+\|D_\sigma I\|)\le C.
$$
This proves the last two bounds in \eqref{eq:raw-entrance}. Along finite-parameter families with $\|\sigma\|_{C^2}=O(\Sigma)$, the inequalities $\Sigma/\eps=o(1)$ also give $\mathfrak d,\mathfrak n=O_{C^2}(\eps^2)$.

For parameter regularity that also allows the terminal energy to vary, use five decreasing source weights
$$
x_j=\delta(1-j/10),\qquad y_j=2-j/10,\qquad 0\le j\le4.
$$
The weighted source components are $(r^\delta F_0,r^2F_1,G_e,G_o)$. For $0<x\le\delta$ and $1<y\le2$, their source norm is
$$
\|I\|_{\mathcal Y_{x,y}}
=\|r^\delta F_0\|_{C^{0,\alpha}_x}
+\|r^2F_1\|_{C^{0,\alpha}_y}
+\|G_e\|_{C^{0,\alpha}_x}
+\|G_o\|_{C^{0,\alpha}_y},
$$
with the stated spherical projections on the last two components. In $\mathcal Y_{x,y}$, the low equations are read with $\beta_0=x$ and $\beta_1=y$, absorbing $r^{\delta-x}$ and $r^{2-y}$ into their bounded source coordinates. The weights above remain positive, with $y_j>1$. Denote the full source update by $\mathcal T$, including the choice of terminal phase and center from Lemma~\ref{lem:fixed-entry} and the prescribed high entrance. Its restriction is a contraction in each $\mathcal Y_{x_j,y_j}$. Terminal derivatives multiply the coefficients by at most polynomial functions of $t$; a strict gap between successive weights absorbs these factors. Difference quotients first give a Lipschitz fixed point in the second space. Expanding its equation gives $(\operatorname{Id}-D_I\mathcal T)I_\mu=\mathcal T_\mu$ in the next space. A second difference quotient gives
$$
(\operatorname{Id}-D_I\mathcal T)I_{\mu\nu}
=\mathcal T_{\mu\nu}+\mathcal T_{\mu I}I_\nu+
\mathcal T_{\nu I}I_\mu+\mathcal T_{II}[I_\mu,I_\nu].
$$
The fourth and fifth spaces give its existence and continuity, respectively. The inverse norms are bounded by the contraction constant. The difference-quotient argument proves $C^2$ dependence in the weaker weights. The positivity estimate follows from the relative ball, and elliptic regularity in the original smooth metric gives smoothness off the puncture.
\end{proof}

\section{The augmented inverse and physical boundary matching}\label{sec:inverse}

We first prove the stronger high-mode estimate used in Proposition~\ref{prop:end}, then construct the local Cauchy inverse and couple these inverses on $M$. Let $\lambda>0$ be the common physical dilation, so that $y=x_i/\lambda$ in a normal chart at $p_i$. The end energy parameters $\eps_i$ will be chosen in terms of $\lambda$ in Section~\ref{sec:projected}. Parameter derivatives in this section keep $\lambda$, the $\eps_i$ and the terminal energies fixed.

\subsection{The local inverse}

The actual trajectory satisfies $w\asymp v_{\eps,e_\infty}(t+\phi_\infty)$. The period estimate after \eqref{eq:wronskian} gives, on every interval $I$,
$$
0\le Qp\,w^{p-1}\le(a+1)(a+2),\qquad
\int_IQp\,w^{p-1}\,dt\le C+C|I|/T_\eps.
$$
For mode $j\ge2$ put $\kappa_j=((a+j)^2-(a+1)(a+2))^{1/2}>0$. The terminal Riccati equation
$$
\rho_j'=\rho_j^2-(a+j)^2+Qp\,w^{p-1},\qquad \rho_j(R)=a+j,
$$
has a backward invariant interval $[\kappa_j,a+j]$. Two solutions with different distant terminal points differ by a quantity that decreases exponentially backwards, since their difference has derivative $(\rho_1+\rho_2)(\rho_1-\rho_2)$. Passing to the limit gives a unique bounded stable solution. For $y_j(t)=\exp(-\int_{t_0}^t\rho_j)$, the exact identity
$$
\int_t^s(a+j-\rho_j)
=\int_t^s\frac{Qp\,w^{p-1}}{a+j+\rho_j}
-\log\frac{a+j+\rho_j(s)}{a+j+\rho_j(t)}
$$
therefore implies
\begin{equation}\label{eq:strong-kernel}
\frac{y_j(s)}{y_j(t)}\le C e^{-(a+j-C/((a+j)T_\eps))(s-t)},\qquad
0\le G_j(t,s)\le\frac C{a+j}e^{-(a+j-C/((a+j)T_\eps))|t-s|}.
\end{equation}
Here the zero entrance kernel is $G_j(t,s)=y_j(t)y_j(s)\int_{t_0}^{\min(t,s)}y_j^{-2}$. The factor $1/(a+j)$ follows from $\rho_j\ge\kappa_j$ and the bounded ratio $(a+j)/\kappa_j$. On $j\ge3$ the corresponding limiting decay rate is $a+3$.

For any fixed $0<\delta_*<2$, \eqref{eq:strong-kernel} gives the high Poisson and zero entrance inverses in absolute cylinder weight $a+\delta_*$, for all sufficiently small $\eps$. Spherical $L^2$ estimates are obtained by the diagonal kernels, and unit-cylinder elliptic and boundary Schauder estimates recover the full $C^{2,\alpha}$ estimates. Finite strictly decreasing weights absorb the polynomial time factors in two parameter derivatives. On every fixed cylinder the coefficients converge with two parameter derivatives to the full-bubble coefficients. Splitting at a fixed large cylinder and using the exponential tail proves convergence of the high Dirichlet-to-Neumann operators in
$$
C^2\mathcal L\big(\Pi_{\ge2}C^{2,\alpha}(S),\Pi_{\ge2}C^{1,\alpha}(S)\big).
$$
In physical variables the limiting solution is $O(r^{\delta_*})$ and its flux tends to zero. Its linear equation therefore extends weakly across the origin, where the coefficients are smooth. Elliptic regularity identifies the limit with the regular inner-ball solution for the full bubble.

Let $B(y)=(2/(1+|y|^2))^a$. Its physical linearized operator is $A_B=-\Delta-QpB^{p-1}$. Under the conformal compactification by $B$, it becomes $-\Delta_{\Sph^n}-n$. Its kernel consists of the $n+1$ first spherical harmonics, transported to Jacobi functions $Z_0,\ldots,Z_n$. Choose a nonnegative cutoff $\chi$ on a fixed annulus and put $\psi_\alpha=\chi Z_\alpha$. The matrix $\mathbf M=(\int\chi Z_\alpha Z_\beta)_{\alpha\beta}$ is positive definite: if the integral of the square of a linear combination is zero, that Jacobi function vanishes on an open annulus and hence everywhere. We solve
\begin{equation}\label{eq:bubble-inverse}
A_B\mathcal R_Bf=f+\sum_{\alpha=0}^nc_\alpha\psi_\alpha,\qquad
(c_\alpha)_\alpha=-\mathbf M^{-1}\left(\int Z_\alpha f\right)_\alpha,
\end{equation}
and add kernel elements so that $\mathcal R_Bf(0)=0$ and $\nabla\mathcal R_Bf(0)=0$. The Jacobi value-gradient matrix is invertible: on the compact sphere it is the map taking a linear function in $\R^{n+1}$ to its value and tangential differential at one point.

For sources a fixed positive distance from the origin this inverse has an absolute kernel bound
\begin{equation}\label{eq:newton-kernel}
|K_B(x,y)|\le C|x-y|^{-m}.
\end{equation}
To verify it, the kernel with the first harmonics removed on $\Sph^n$ is bounded by $C(d(x,y)^{-m}+1)$. Multiplication by the two conformal factors and the stereographic chordal distance identity give the Newton bound. The compensation and the value-gradient correction add finite-rank kernels bounded by $C\langle x\rangle^{-m}\langle y\rangle^{-m}$, which is at most $C|x-y|^{-m}$. Source separation is needed when evaluating the kernel and its derivative at the origin.

Replace the operator inside the fixed small sphere $r_{\rm in}=e^{-t_0}$ by the actual high tail operator. Let $H_B^-$ and $H_B^+$ be the full-bubble high-mode Poisson operators on $t\le t_0$ and $t\ge t_0$, respectively, with the prescribed trace at $t_0$. Their physical representatives $r^{-a}H_B^-\sigma$ and $r^{-a}H_B^+\sigma$ decay at spatial infinity and extend regularly across the origin, respectively. Set $D_B^-=+\left.\partial_tH_B^-\right|_{t=t_0}$, $D_B^+=-\left.\partial_tH_B^+\right|_{t=t_0}$ and $D_w^+=-\left.\partial_tH_w^+\right|_{t=t_0}$. The sum $D_B^-+D_B^+$ is invertible on high data by the full-bubble inverse. Its perturbation $D_B^-+D_w^+$ remains invertible by the operator convergence just proved. If
$$
\eta_B(\theta,y)=r_{\rm in}^aK_B(r_{\rm in}\theta,y),\qquad
\delta_w=-(D_B^-+D_w^+)^{-1}(D_w^+-D_B^+)\eta_B,
$$
the new kernel is $K_B+r^{-a}H_B^-\delta_w$ outside and $r^{-a}H_w^+(\eta_B+\delta_w)$ inside. Its two value traces agree. Its derivative matching equation is exactly
$$
-D_B^+\eta_B+D_B^-\delta_w=-D_w^+(\eta_B+\delta_w).
$$
It preserves the low compensation and the zero value-gradient convention and satisfies \eqref{eq:newton-kernel} uniformly in small $\eps$.

These estimates are uniform on compact sets of full-bubble center and scale parameters. For each full bubble, use its own Jacobi fields in $\psi_\alpha=\chi Z_\alpha$, with the same fixed cutoff separated from the puncture. The value and gradient normalization is always imposed at the fixed physical puncture $y=0$; the compensation Gram matrices and the Jacobi value-gradient inverses are uniformly bounded on these compact parameter sets. The conformal identity
$$
|T_bx-T_by|^2=\Omega_b(x)\Omega_b(y)|x-y|^2
$$
preserves the Newton kernel after multiplication by its two conformal weights. Also $\Omega_b^a B_\ell(T_bx)=(2\ell/(\ell^2(1-2b\cdot x+|b|^2|x|^2)+|x|^2))^a$ is a smooth full bubble even at the apparent pole of $T_b$. Use the compact sphere to differentiate this family. For two parameter derivatives, the resolvent identities insert at most two potentials bounded by $C\langle x\rangle^{-4}$; the Newton product estimate below bounds these insertions. At a moving inner interface, pull back both fields and traces to the fixed domains before taking these derivatives.

For general Cauchy discrepancies $\mathfrak d\in C^{2,\alpha}$ and $\mathfrak n\in C^{1,\alpha}$, use cylindrical representatives in the unit-bubble chart. Under $t=-\log r$, conformal covariance gives $r^{a+2}A_Bz=-\mathcal L_{v_0}(r^az)$. In the following Cauchy calculation, multiply the physical fields and Jacobi functions by $r^a$, and the sources and compensation functions by $r^{a+2}$, retaining their letters. Since $n-2a-2=0$, the source--Jacobi integral pairing is unchanged. The exterior equation and interface conditions become
$$
-\mathcal L_{v_0}z=f+\sum c_\alpha\psi_\alpha,\qquad
z(t_0)=\sigma+\mathfrak d,\qquad z'(t_0)-(H_w^+\sigma)'(t_0)=\mathfrak n.
$$
The low Wronskians give, for each regular low Jacobi function $Z_j$,
\begin{equation}\label{eq:low-cauchy}
\sum_\alpha c_\alpha\int_{-\infty}^{t_0}Z_j\psi_\alpha
=-Z_j(t_0)\mathfrak n_j+Z_j'(t_0)\mathfrak d_j-\int_{-\infty}^{t_0}Z_jf.
\end{equation}
Choose $t_0$ beyond the bubble maximum for the entire compact parameter set. Then the low regular functions do not vanish at the entrance. The low values fix their homogeneous coefficients. If $\ell(f+\sum_\alpha c_\alpha\psi_\alpha)$ is the normal trace of the zero entrance outer particular solution, the remaining equation is
\begin{equation}\label{eq:high-cauchy}
(D_B^-+D_w^+)\sigma=\mathfrak n_{\ge2}-\ell(f+\sum_\alpha c_\alpha\psi_\alpha)-D_B^-\mathfrak d_{\ge2}.
\end{equation}
This solves both traces and all compensation coefficients. The physical interpretation uses \eqref{eq:physical-conormal}; after equality of the value traces its tangential and Robin terms cancel. The difference between the moving and static conormals, including $h_t(t_0)$, is part of $\mathfrak n$.

\subsection{Coupling the local inverses}

We now couple the local inverses through the Green function of a fixed auxiliary metric. Flatten $g$ in normal balls of radius $r_0$, leaving it unchanged outside the balls of radius $2r_0$, and call the auxiliary metric $g_0$. As proved in Section~2 its positive inverse has uniform Green estimates. First truncate the outer bubble potentials at scaled radius $L$. Their omitted part has Newton integral norm and scaled H\"older multiplier norm at most $CL^{-2}$. Iterating the local inverse thus gives an inverse for the compactly supported potentials with bounds independent of sufficiently large $L$. Write $V_i^L$ for the truncated potential in the $i$th scaled chart, supported in $B_{2L}$. Denote the corresponding augmented solution by $\mathcal R_i^Lf$ and its compensation coefficients by $c_{i\alpha}^L(f)$, with the same compensation functions, inner high-mode tail and value-gradient normalization as in the preceding local construction.

The coupling can be written as an equation for the harmonic fields entering the local balls. In the $i$th scaled chart, take $h_i$ harmonic on $B_{6L}$. Choose a cutoff $\chi_i$ equal to zero near the inner end and one outside a fixed scaled annulus, with no outer cutoff. For $A_i^L=-\Delta-V_i^L$, define $f_{h_i}=A_i^L(\chi_i h_i)$ on $B_{6L}$ and extend it by zero. Since $V_i^L$ is supported in $B_{2L}$, this source vanishes on a neighborhood of $\partial B_{6L}$ and is supported between a fixed inner radius and $2L$. Extend $(\chi_i-1)h_i$ by zero as well, and put
$$
v_{h_i}=(\chi_i-1)h_i-\mathcal R_i^Lf_{h_i},\qquad
u_{h_i}=h_i+v_{h_i}\quad\hbox{on }B_{6L}.
$$
Thus $v_{h_i}=-\mathcal R_i^Lf_{h_i}$ outside $B_{2L}$, and
$$
j_{h_i}=-\Delta v_{h_i}=V_i^Lu_{h_i}-\sum c_{i\alpha}\psi_{i\alpha}.
$$
This density is compactly supported and integrable at the end, by the strong high-mode estimate. The decaying field $v_{h_i}$ is its Newton potential. Put
$$
B_m(L)=\begin{cases}1,&m<2,\\1+\log L,&m=2,\\L^{m-2},&m>2.\end{cases}
$$
Radial Newton integration gives $\|j_{h_i}\|_{L^1}\le CB_m(L)\|h_i\|_{L^\infty(B_{6L})}$. For a prescribed local source $f_i$, the scattering density is
$$
j_{f_i}=f_i+V_i^L\mathcal R_i^Lf_i+\sum_\alpha c_{i\alpha}^L(f_i)\psi_{i\alpha},
$$
and its $L^1$ norm is at most $C\|f_i\|_1$. The constants are independent of $L$: one uses $\int V_i^L|\mathcal R_i^Lf|\le C\int|f|$, obtained by exchanging the Newton integrals; at the inner end the estimate is $\int_0^{r_{\rm in}}r^{n-3+\delta_*}\,dr<\infty$.

Split the physical source into local pieces $f_i$ supported in $B_{9L\lambda}(p_i)$ and an outer source $f_{\rm out}$ that vanishes in $B_{8L\lambda}(p_i)$. Let $G_0(x,y)$ be the integral kernel of $P_{g_0}^{-1}$. In a flat ball write
$$
G_0(x,y)=(m\omega)^{-1}|x-y|^{-m}+H_i(x,y).
$$ The incoming fields solve the explicit system
\begin{equation}\label{eq:global-coupling}
h_i(x)=P_{g_0}^{-1}f_{\rm out}(x)
+\int H_i(x,y)j_i(y)\,dy
+\sum_{j\ne i}\int G_0(x,y)j_j(y)\,dy,
\end{equation}
where $j_i=j_{f_i}+j_{h_i}$. Each right side is harmonic on the $i$th scaled ball $B_{6L}$, which contains the support of $V_i^L$. The diagonal feedback has norm at most $C(\lambda/r_0)^mB_m(L)$ and the off-diagonal feedback at most $C\lambda^mB_m(L)$. Fix $r_0,L$ and then take $\lambda$ small; the row sums are less than $1/4$. A Neumann series solves \eqref{eq:global-coupling}. The field $P_{g_0}^{-1}f_{\rm out}+\sum_iP_{g_0}^{-1}j_i$ agrees in each local ball with the core response used to define $j_i$. Its traces therefore match, and the global equation has exactly the $n+1$ compensation terms at each end.

General Cauchy data are added using \eqref{eq:low-cauchy}--\eqref{eq:high-cauchy}. Their compactly supported distribution is tested in \eqref{eq:global-coupling} by its exact boundary pairing
$$
\langle j,H\rangle=\int_{\partial B_s}(v\partial_\nu H-H\partial_\nu v).
$$
The fixed outer sphere lies away from all moving interfaces. This formula also proves twice differentiable parameter dependence, without asserting that derivatives of moving interface distributions are $L^1$ densities.

\subsection{Weighted estimates}

We use separate norms for far sources, near sources and Cauchy data. All near collars are fixed in scaled static coordinates. Choose a fixed scaled radius $r_*>0$ below the compensation supports and all outer coordinate buffers, and then take $r_{\rm in}<r_*/8$. A cutoff $\chi_{\rm near,i}$ equals one between the entrance and radius $r_*$, and vanishes beyond radius $2r_*$. Split the source itself as
$$
f= f_{\rm far}+\sum_i f_{{\rm near},i},\qquad
f_{{\rm near},i}=\chi_{{\rm near},i}f.
$$
The far part is extended by zero across the inner balls. All collars and interface data are read in fixed static charts; moving charts are pulled back to these fixed domains before parameter differentiation.

For a positive source envelope $F_*$ let $E=P_{g_0}^{-1}F_*$. The envelope is supported outside scaled radius $4r_*$ and includes, for each end, a smooth nonnegative test density of physical size $\lambda^{-a-2}\Sigma_i$, supported in a fixed separated scaled annulus. This density records the near-source norm. Assume
$$
\lambda^a E(p_i)\le C\Sigma_i.
$$
Since $E$ is harmonic in the flat near balls, its scaled value there is bounded above by $C\Sigma_i$; the test density gives the lower bound $C^{-1}\Sigma_i$. We use these pointwise comparisons in the source-free near balls. In the remaining outer region define
\begin{equation}\label{eq:source-norm}
\|f\|_{Y_E}=\sup_x\frac{P_{g_0}^{-1}|f|(x)}{E(x)}
+\sup_x\frac{\ell(x)^2\|f\|_{C^{0,\alpha}_{\scnorm}(B_x)}}{E(x)}.
\end{equation}
The first supremum in \eqref{eq:source-norm} is over $M$, and its second supremum is over the allowed far-source region. Here $\ell(x)$ is a small fixed multiple of the distance to the punctures, capped by a fixed number, and $B_x=B_{\ell(x)}(x)$. The scaled H\"older norm uses powers of $\ell(x)$ for derivatives and the H\"older seminorm. Overlapping near collars are chosen so that the corresponding balls in the outer region covered by the solution norm avoid the interfaces. Define the outer solution norm by
$$
\|z^-\|_{X_E}=\sup_x
\frac{\|z^-\|_{C^{2,\alpha}_{\scnorm}(B_x)}}{E(x)}.
$$
The near-boundary norms are included separately in the product space below. In the original fixed normal coordinates $x_i$ put
$$
\widehat f_i(y)=\lambda^{a+2}f_{{\rm near},i}(x_i^{-1}(\lambda y)),\qquad
\widehat z_i^-(y)=\lambda^a z^-(x_i^{-1}(\lambda y)).
$$
Their near-collar norms are taken after the bounded static coordinate pullback, on fixed annuli outside the entrance. The complete source norm is
$$
\|f\|_{\mathbb Y_E}
=\|f_{\rm far}\|_{Y_E}
+\max_i\Sigma_i^{-1}\|\widehat f_i\|_{C^{0,\alpha}}.
$$
For $c_i=(c_{i0},\ldots,c_{in})$ and the returned high trace $\sigma_i$, let $z_i^+$ be the inner cylindrical field. The corresponding static outer cylindrical field is
$$
z_i^-(t,\theta)=(rA_{\tau_i,b_i})^a
\widehat z_i^-(K_{\tau_i,b_i}(r\theta)),\qquad r=e^{-t},
$$
with $b_i$ the fixed entrance center; the inner field uses the actual moving map. Define
\begin{align*}
\|(z,\sigma,c)\|_{\mathbb X_E}
={}&\|z^-\|_{X_E}+\max_i\Sigma_i^{-1}
\bigl\{\|\widehat z_i^-\|_{C^{2,\alpha}}+\|\sigma_i\|_{C^{2,\alpha}}+|c_i|\\
&\hspace{31mm}+\|\Pi_{e,\ge2}z_i^+\|_{C^{2,\alpha}_{a+1}}
+\|\Pi_{o,\ge3}z_i^+\|_{C^{2,\alpha}_{a+2}}\bigr\}.
\end{align*}
The Cauchy data $\mathfrak d_i,\mathfrak n_i$ are in the cylinder units of \eqref{eq:low-cauchy}--\eqref{eq:high-cauchy}; their norm is
$$
\max_i\Sigma_i^{-1}
\bigl(\|\mathfrak d_i\|_{C^{2,\alpha}}+\|\mathfrak n_i\|_{C^{1,\alpha}}\bigr).
$$
The fixed coordinate and conformal conversions to physical Cauchy data are bounded in these norms. The same $\Sigma_i$ normalizes the near source, Cauchy data, returned high trace and compensation coefficients at end $i$.

For general $E$ in the outer region the needed comparison is the mean bound $\fint_{B_s(x)}E\le CE(x)$, proved by splitting the positive Green integral into near and far parts. In each smaller ball this bound selects a center $z$ with $E(z)\le CE(x)$ and $\ell(z)\asymp\ell(x)$. The centered norm at $z$ controls a ball containing that smaller ball; a finite covering therefore compares the centered norms with different fixed radius factors. Local $L^1$ elliptic estimates followed by Schauder estimates give the outer $X_E$ bound for a far source. The same norm controls the compensation sources: source separation gives $E(x)\ge C^{-1}E(p_i)$ on each fixed scaled compensation annulus, and their Newton potentials are bounded by $CE$. Near sources and Cauchy data are handled by the local lifts constructed in the proof below.

Write $\mathcal Nf(x)=\int_{\R^n}|x-y|^{-m}f(y)\,dy$ for the Euclidean Newton potential. The elementary three-kernel inequality
$$
|x-z|^{-m}|z-y|^{-m}\le C|x-y|^{-m}
(|x-z|^{-m}+|z-y|^{-m})
$$
implies $\mathcal N(V\mathcal N\mu)\le C\|\mathcal NV\|_\infty\mathcal N\mu$ for nonnegative data. If a principal coefficient error has scaled coefficients bounded by $C\tau r^d$ through order two, its map $X_E\to Y_E$ is bounded by $C\tau R^d$ on $r<2R$. This follows by bounding it by $C\tau r^{d-2}E\|z\|_{X_E}$, applying the three-kernel inequality, and using the centered radius comparison for its H\"older term. Thus the true geometric principal coefficients, and the omitted outer potential, are small perturbations of the global inverse.

To keep the constants compatible with the later nonlinear construction, choose the fixed parameters in the following order. Use first the inverse $\mathcal R_B$ without an artificial entrance to choose $r_0$ small and $L$ large. Its far-source constants and separated-core jet constants do not depend on $t_0$. For each subsequently fixed $t_0$, the actual far-source inverse converges to it by the explicit boundary resolvent above. Moreover zero value and gradient improve its inner boundary kernel to $C r_{\rm in}^{a+2}\langle y\rangle^{-m}$. The strong Poisson estimate then gives
$$
|K_w(x,y)|\le C r_{\rm in}^{2-\delta_*}r^{\delta_*}\langle y\rangle^{-m},\qquad
\int_{r<r_{\rm in}}V|K_w(x,y)|\,dx\le C r_{\rm in}^{n}\langle y\rangle^{-m}.
$$
These bounds allow the far constants to be fixed before $t_0$. Near-source and Cauchy constants may depend on the now fixed $t_0$. The near principal error is $O(\tau)$; elimination of that block adds only $O(C(t_0,r_0,L)^3\tau)$ to the far equation. It tends to zero when $\lambda\to0$. Thus $r_0$ and $L$ remain fixed when $t_0$ is chosen.

In the following energy estimate, the squared norm on $M$ is
$$
\|z\|_{H^1(M)}^2+\sum_i\int_{B_{r_0}(p_i)}r_i^{-2}z^2.
$$
On a scaled Euclidean patch we use the energy
$$
\int_{\R^n}\bigl(|\nabla z|^2+\langle y\rangle^{-2}z^2\bigr),
$$
without requiring a full Euclidean $L^2$ tail. For a far volume source with zero prescribed jumps, the signed augmented inverse satisfies
\begin{equation}\label{eq:inverse-energy}
\|\mathcal Rf\|_{H^1\cap\mathrm{Hardy}}
\le C\|f\|_{Y_E}\left(\int F_*P_{g_0}^{-1}F_*\right)^{1/2}.
\end{equation}
For the base inverse, its absolute kernel bound, the equation tested against the solution, local Hardy inequalities and $\lambda^m(P_{g_0}^{-1}|f|(p_i))^2\le C\int|f|P_{g_0}^{-1}|f|$ prove the estimate. The last point bound follows from the mean estimate in a source-free core and the Sobolev inequality. For the principal perturbation, sum its convergent source series in $H^1$. Indeed if $P_{g_0}^{-1}|f|\le YE$, symmetry and Tonelli give $\int|f|P_{g_0}^{-1}|f|\le Y^2\int F_*E$. This proves \eqref{eq:inverse-energy} for the perturbed signed inverse using the source series.

\begin{proposition}\label{prop:global-inverse}
Let $A^-$ be the exterior operator obtained from $P_{g_0}-\sum_iV_i^{L,\mathrm{phys}}$ by the far and near coefficient perturbations described above. In the fixed compensation supports set $\Psi_{i\alpha}(x)=\lambda^{-a-2}\psi_{i\alpha}(x_i/\lambda)$; the coordinate and conformal extensions are the identity there. With the spaces just defined, the augmented construction gives a linear map
$$
\mathcal R_w^{\rm glob}(f;\mathfrak d,\mathfrak n)=(z,\sigma,c)
$$
solving
$$
A^-z^-=f+\sum_{i,\alpha}c_{i\alpha}\Psi_{i\alpha},\qquad
-\mathcal L_{w_i}z_i^+=0,\qquad \Pi_{0,1}z_i^+=0,
$$
with $z_i^-(t_0)-z_i^+(t_0)=\mathfrak d_i$, $(z_i^-)'(t_0)-(z_i^+)'(t_0)=\mathfrak n_i$, and $z_i^+(t_0)=\sigma_i$. There are exactly $n+1$ compensation coefficients at each end. For sufficiently small $\lambda$ and $\eps_i$, it satisfies
$$
\|(z,\sigma,c)\|_{\mathbb X_E}
\le C\left\{\|f\|_{\mathbb Y_E}
+\max_i\frac{\|\mathfrak d_i\|_{C^{2,\alpha}}+\|\mathfrak n_i\|_{C^{1,\alpha}}}{\Sigma_i}\right\}.
$$
The constant may depend on the fixed compact parameter set and on $r_0,L,t_0,r_*$, and is independent of the small scales. The map has joint derivatives through order two in the finite entrance parameters and the tail source $I\in\mathcal Y_{x_4,y_4}$, in the fixed-domain topology and with the same bounds. The scales and terminal energies are fixed in these derivatives. For a far volume source with zero jumps its energy is controlled by \eqref{eq:inverse-energy}. A near source or Cauchy datum of scaled size $\Sigma_i$ has piecewise squared Dirichlet--Hardy energy at most $C\Sigma_i^2$.
\end{proposition}

\begin{proof}
For far sources, combine the local kernels with \eqref{eq:global-coupling}. The resulting field satisfies the value and centered Schauder estimates proved above. Source separation and the bound $\lambda^aE(p_i)\le C\Sigma_i$ give the scaled near-collar norms, high traces and compensation coefficients. The strong high estimates give the two inner norms. Finite coordinate pullbacks preserve these bounds on the fixed parameter set.

For a near source solve the local Cauchy problem \eqref{eq:low-cauchy}--\eqref{eq:high-cauchy} with zero jumps, using its actual $C^{0,\alpha}$ norm. For prescribed jumps solve the same problem with zero volume source. Denote either resulting local lift by $Q_i$. Its outer field decays as $\langle y\rangle^{-m}$; its near and inner high norms and its compensation vector are bounded by the near-source or Cauchy norm. To truncate the outer potential at $L$, subtract the zero-jump response to the omitted potential acting on $Q_i$. That source is $O(\langle y\rangle^{-m-4}\mathbf1_{r>L})$ times the datum size. Newton integration bounds its response by $CL^{-2}\langle y\rangle^{-m}$. Thus the modified lift preserves both prescribed jumps and all the asserted bounds.

Its distribution $-\Delta Q_i$ is supported in the $2L$ ball. For Cauchy jumps it includes the interface term, whose action on a smooth test $H$, with a common normal direction, is
$$
\int_S\bigl(-[\partial_\nu Q_i]H+[Q_i]\partial_\nu H\bigr).
$$
The volume part is integrable by the strong inner estimates and the outer decay. The distribution has bounded action on the harmonic tests used in \eqref{eq:global-coupling}; its bound is the datum size. Insert these known local distributions into that same coupling system. The feedback operator is unchanged. It follows that a local datum of size $d_i$ contributes at most $Cd_i$ at its own core, $C\lambda^md_i$ at each other core, and physical charge $C\lambda^ad_i$. The annular test density in $F_*$ controls its outer field and local norms. This proves the product-space estimate, including the boundary and compensation norms. Parameter derivatives are taken on the fixed outer sphere in the boundary pairing, so no differentiability of moving interface measures as $L^1$ densities is asserted.

Restore far principal coefficients by their small map $X_E\to Y_E$. Restore the near coefficient error in the near $C^{2,\alpha}\to C^{0,\alpha}$ block; its norm is $O(C(t_0)\tau_i)$. Eliminate that block first. Its correction to the far equation has size $O(C(t_0,r_0,L)^3\max_i\tau_i)$ and tends to zero with the scales. The far small constants were fixed before $t_0$, as proved above. Hence this inversion preserves the earlier choices of $r_0,L$. Sum the compensation vectors along with the fields in the convergent source series. Differentiating the local boundary resolvents, the coupling equation and these small inverses twice proves the stated joint parameter regularity. The positive source weights control the decaying trajectory variations, while the terminal energy is fixed, so the resulting coefficient derivatives are bounded in the displayed inner weights.

For near and boundary columns, the local piecewise energy is bounded by the square of the scaled datum norm: the outer field has decay $r^{-m}$ and the inner high field has the displayed stronger decay. The physical rescaling preserves the gradient and Hardy energies because $n-2a-2=0$. The smooth global feedback has smaller energy, with the same bound. Principal corrections are summed in these energy norms just as in \eqref{eq:inverse-energy}. A lift with a nonzero value jump is treated piecewise; after its jump is cancelled by the reference, the resulting field is in the global $H^1$ space.
\end{proof}

\section{A global family solving the projected equation}\label{sec:projected}

We now fix the energy parameters. For a common $\lambda\to0$, put
\begin{equation}\label{eq:energy-choice}
\eps_i^2=\begin{cases}
\lambda^{2d_i},&i\in\mathcal C_<,\\
\lambda^m\log(1/\lambda),&i\in\mathcal C_=,\\
\lambda^m,&i\in\mathcal J.
\end{cases}
\end{equation}
The geometric amplitude is $\tau_i=\lambda^{d_i}$ on $\mathcal C$ and $\tau_i=\lambda^q$ on $\mathcal J$. These choices satisfy $\tau_i=O(\eps_i)$ and $\tau_i\eps_i^{-4/n}\to0$. At a mass point the exponent is $q-4a/n>0$; for $n=3,4$ it is $4/3,1$, respectively. We keep $\lambda$ and the $\eps_i$ fixed when differentiating in the finite-dimensional parameters.

Parametrize the $i$th full bubble by $s_i\in\R$ and $\zeta_i\in\R^n$ in the compact set chosen in Section~\ref{sec:reduction}. Its width is $\lambda e^{s_i}$, its center is $\lambda e^{s_i}\zeta_i$, and its monopole coefficient is $\alpha_i=(2\lambda e^{s_i})^a$. The corresponding actual entrance coordinates are
\begin{equation}\label{eq:parameter-change}
\phi_i=s_i+\log(1+|\zeta_i|^2),\qquad
b_i=-e^{-s_i}\frac{\zeta_i}{1+|\zeta_i|^2}.
\end{equation}
Their inverse is $\zeta_i=-e^{\phi_i}b_i$, $s_i=\phi_i-\log(1+e^{2\phi_i}|b_i|^2)$. We impose $\phi_i(t_0)=\phi_i$ and $b_i(t_0)=b_i$ by Lemma~\ref{lem:fixed-entry}, separately for each actual source. Thus the terminal phase and center vary with the source; the static bubble and its static normal chart depend only on $(s_i,\zeta_i)$. The change of variables is smooth on each fixed compact parameter set; $\zeta_i$ moves the bubble center while the puncture $p_i$ remains fixed.

\subsection{The complete first response}

Interpolate from the flattened metric $g_0$ to $g$ by a family $g_t$, $0\le t\le1$, equal to $\exp(t\log g)$ in the flat local balls and joined to $g$ by a fixed cutoff outside. The quadratic-form estimate in Section~2 keeps $P_{g_t}$ positive along this family. Write
$$
A_t=\sum_i\alpha_iG_{g_t}(\cdot,p_i),\qquad
R_2=A_1-A_0-\dot A_0.
$$
Write $P_t=P_{g_t}$ and, for each pole, $G_t=G_{g_t}(\cdot,p_i)$; dots denote differentiation in $t$. Differentiating the Green equation gives $P_t\dot G_t=-\dot P_tG_t$ and $P_t\ddot G_t=-\ddot P_tG_t-2\dot P_t\dot G_t$. In a ball with $g-\delta=O_{\scnorm}(r^d)$, the first operator derivative has coefficient orders $(r^d,r^{d-1},r^{d-2})$, and the second has $(r^{2d},r^{2d-1},r^{2d-2})$. The same Newton decomposition used in Lemma~\ref{lem:mass} therefore show that
$$
G_{g_t}-r^{-m},\ \dot G_{g_t}=O_{\scnorm}(\Psi_d),\qquad
\ddot G_{g_t},\ G_g-G_{g_0}-\dot G_{g_0}=O_{\scnorm}(\Psi_{2d}),
$$
where $\Psi_\beta=1+r^{\beta-m}$ for $\beta<m$, $1+|\log r|$ for $\beta=m$, and $1$ for $\beta>m$. The regular part of the first Green variation contributes a second-variation source of size $r^{d-2}(1+|\log r|)$. Its Newton potential is bounded, so it is included in the estimates above.

Let $B_i^{\rm phys}$ be the full physical bubble. Choose smooth cutoffs $0\le\widehat\chi_i\le1$ in the original normal charts, equal to one on $B_{r_0}(p_i)$ and zero outside $B_{2r_0}(p_i)$, independently of $\lambda$ and the finite bubble parameters. The commutators of these cutoffs with the bubble-minus-monopole terms are supported in fixed physical annuli and are $O_{C^2}(\lambda^{a+1})$ by the expansion below; they are included in the compact-region source bounds. Starting from $A_0+\sum_i\widehat\chi_i(B_i^{\rm phys}-\alpha_i r_i^{-m})$, remove the smooth regular part $A_0-\alpha_ir_i^{-m}$ on the fixed scaled inner collars with a second cutoff. The resulting positive reference $W_0$ equals $B_i^{\rm phys}$ on those collars and retains the original expression elsewhere. This modification costs $O(\lambda^m)$ in scaled source norms, also when $n=3$.

In the static actual normal chart, suppress the end index and put $K=K_{\tau,b}$, $A=A_{\tau,b}$ and $r=e^{-t}$. The pullbacks of physical fields and sources include the common dilation:
\begin{align*}
\mathscr P u&=(\lambda rA)^a u\bigl(x_i^{-1}(\lambda K(r\theta))\bigr),\\
\mathcal D f&=(\lambda rA)^{a+2}f\bigl(x_i^{-1}(\lambda K(r\theta))\bigr).
\end{align*}
Let $B(t)=v_0(t+\phi_i)$ denote the cylindrical bubble in this subsection. At $\tau=0$, we have $K=T_b$ and $A=\Omega_b$, and direct substitution of \eqref{eq:parameter-change} gives
$$
\mathscr P B_i^{\rm phys}
=\left(\frac{2e^{\phi_i}r}{e^{2\phi_i}+r^2}\right)^a=B(t).
$$
The cylindrical linearization is $-\mathcal L_B=-\partial_t^2-\Delta_S+a^2-QpB^{p-1}$. At $\tau=0$, conformal covariance gives the exact identity
$$
(-\mathcal L_B)(\mathscr P u)
=\mathcal D\bigl((P_\delta-Qp(B_i^{\rm phys})^{p-1})u\bigr),
$$
where $P_\delta$ acts in the physical normal coordinate $x_i$. These maps and $B$ are independent of the tail source $I$. The exact reference difference
$$
B_{\rm ref}=\mathscr PW_0-B
$$
is generally $O(\tau)$, rather than quadratic. Near the entrance choose the auxiliary outer operator $A^-=\mathcal D^{-1}(-\mathcal L_B)\mathscr P$, and away from a fixed near buffer choose $A^-=P_g-QpW_0^{p-1}$. Interpolate their coefficients in a fixed buffer contained in the region where $W_0=B_i^{\rm phys}$. There the static cylinder metric differs from $dt^2+\gamma$ by $O(\tau)$ in $C^{2,\alpha}$, and $\mathscr PW_0-B=O(\tau)$ in the same norm, including two entrance-parameter derivatives. Conformal covariance and smoothness of the power on the positive bubble range therefore bound the operator difference by $C\tau$ from $C^{2,\alpha}$ to $C^{0,\alpha}$ in the scaled norms. It is handled by the near-source part of the inverse.

Let $f_{\rm fin}=\mathscr S_bB$ be the complete finite-normal geometric first source on the static cylindrical bubble. On the forced tail its counterpart is $f^+[I]=\mathscr S_{b[I]}w[I]$. Prescribe the outer source to be $\mathcal D^{-1}(f_{\rm fin}+\mathcal L_BB_{\rm ref})$ near the collar and
\begin{equation}\label{eq:outer-first-source}
f^-=-\dot P_0W_0+(P_g-P_0)\dot A_0
\end{equation}
in the exterior reference region, with a fixed partition in between. These sources are specified before solving for a response.

Choose a local static particular solution $\Phi_{\rm loc}$ of $-\mathcal L_B\Phi_{\rm loc}=f_{\rm fin}$ in the collar, with the compensation support farther out. Take an outer lift equal to $\mathscr P^{-1}(\Phi_{\rm loc}-B_{\rm ref})$ near the collar and $\dot A_0$ farther out; call it $D^-$. On the collar these definitions give exactly
$$
A^-D^-=\mathcal D^{-1}\{-\mathcal L_B(\Phi_{\rm loc}-B_{\rm ref})\}
=\mathcal D^{-1}(f_{\rm fin}+\mathcal L_BB_{\rm ref})=f^-.
$$
Let $D^+$ be the tail particular solution, including the finite remainder low modes. Apply the actual global Cauchy inverse of Section~\ref{sec:inverse} to source $-(A^-D^--f^-)$ and jumps
$$
D^+-(\mathscr PD^-+B_{\rm ref}),\qquad
(D^+)'-(\mathscr PD^-+B_{\rm ref})'.
$$
Call the correction $J$, and put $Y^-=D^-+J^-$, $Y^+=D^++J^+$. Adding the lifts and the correction gives
\begin{equation}\label{eq:first-response-match}
\begin{gathered}
A^-Y^-=f^-+\sum_\alpha c_{{\rm geom},\alpha}\psi_\alpha,\qquad -\mathcal L_wY^+=f^+[I],\\
\mathscr P(W_0+Y^-)=B+Y^+
\end{gathered}
\end{equation}
with equality of the model derivatives as well. Since the entrance parameters were fixed before constructing the outer data and $A^-$, these objects are independent of $I$. Differentiating the outer equation therefore varies only the response and its compensation coefficients. The terminal parameters remain source-dependent and are included in the tail response. The remaining physical conormal discrepancy is included as in \eqref{eq:physical-conormal} and is $O_{C_I^2}(\eps_i^2)$ by \eqref{eq:entry-cauchy}.

The exterior reference is
$$
W_{\rm aff}^-=W_0+Y^-+\chi_{\rm out}R_2,
$$
where $\chi_{\rm out}$ vanishes on the model collar. In the true exterior it is exactly
$$
U_0+J^-,\qquad U_0=A_1+\sum_i\widehat\chi_i(B_i^{\rm phys}-\alpha_i r_i^{-m}).
$$
The source in \eqref{eq:outer-first-source} gives the cancellation
$$
A^-\dot A_0-f^-=\dot P_0(W_0-A_0)-QpW_0^{p-1}\dot A_0.
$$
The addition of $R_2$ therefore restores the full Green reference while preserving the first-source cancellation. In particular, on a single exterior collar the exact reference identity is
\begin{equation}\label{eq:outer-identity}
P_gU_0-QU_0^p
=(P_g-P_\delta)(B_i^{\rm phys}-\alpha_i r_i^{-m})
+Q\{(B_i^{\rm phys})^p-(B_i^{\rm phys}+A_1-\alpha_i r_i^{-m})^p\}.
\end{equation}
Here $P_gA_1=0$ away from the punctures.

In scaled fixed cores the first response has amplitude $O(\tau)$ on $\mathcal C$. On $\mathcal J$ its particular tail still has amplitude $\tau$, but its homogeneous entrance may have the larger amplitude
$$
\Theta_i\le C\begin{cases}
\lambda^q,&q<m,\\
\lambda^m(1+|\log\lambda|),&q=m,\\
\lambda^m,&q>m.
\end{cases}
$$
In each case $\Theta_i=o(\eps_i)$. In the moving-coordinate error, the particular field retains amplitude $\tau_i$, while the homogeneous field is estimated separately with amplitude $\Theta_i$. On $\mathcal C$, the complete response has natural entrance orders $O(\tau r^{a+2})$ in even high modes and $O(\tau r^{a+3})$ in odd high modes. To obtain these orders with an entrance-independent constant, first use the zero-neck complete-bubble inverse without a cut. The full static correction, after addition of $B_{\rm ref}$, has zero value and gradient and bounded separated-core jets. Taylor's theorem gives these two orders. Only then choose the inner entrance. The actual tail response converges to this one there by the high boundary resolvent. Thus the complete $O(\tau)$ first response is matched before choosing the additional $o(\eps)$ high entrance, including when $\tau\asymp\eps$.

\subsection{Source bounds and the choice of radii}

We choose the transition radii and source weights so that the remaining sources satisfy the product-space estimates and have negligible quadratic energy. Put $L_\lambda=1+|\log\lambda|$ and
\begin{align*}
\gamma&=\frac12\left[\max\left\{\frac1{q+1},\frac1{m+1}\right\}+\frac1{a+1}\right],&
b&=\frac12\min\{1,m\gamma\},\\
\rho_M&=\lambda^\gamma,&\Sigma_0&=\lambda^{a+b}.
\end{align*}
Since $q>a$ and $m>a$, we have $(q+1)\gamma>1$, $(m+1)\gamma>1$, $(a+1)\gamma<1$ and $0<b<\min(1,m\gamma)$. If $\mathcal C$ is nonempty, choose
$$
\theta=\frac12\left(1+\max\left\{\frac13,1-\frac4a,1-m\gamma,\frac{1-\gamma}{m}\right\}\right),\qquad
\rho_C=\lambda^\theta,
$$
and $s_* =\min\{1/2,(1-\theta)(a+1)/(4a)\}>0$. Each lower bound inside the maximum is strictly less than one. Set $\Sigma_i=\eps_i^{1+s_*}$ on $\mathcal C$, and $\Sigma_i=\lambda^m$ on $\mathcal J$, enlarged to $\lambda^mL_\lambda$ in dimension four. Set $\rho_i=\rho_C$ for $i\in\mathcal C$ and $\rho_i=\rho_M$ for $i\in\mathcal J$. These choices satisfy \eqref{eq:sigma-condition}.

The scaled sources of the unknown correction have the following positive envelopes, with smooth cutoffs on the indicated proportional annuli:
\begin{enumerate}
\item $\Sigma_i$ times a smooth density in a fixed scaled annulus outside the near collar, recording the near-source and boundary norms of Proposition~\ref{prop:global-inverse};
\item $\tau_i^2r^{2d_i-m-2}$ on $1<r<2\rho_i/\lambda$, with $d_i=q$ for $i\in\mathcal J$;
\item $K_i^{\rm in}r^{-4}$ on the same region and $K_i^{\rm in}(\rho_i/\lambda)^{-2}$ on $\rho_i/(2\lambda)<r<2\rho_i/\lambda$;
\item in physical exterior collars, $\lambda^{a+1}r_i^{d_i-m-3}$, $\lambda^{a+2}r_i^{d_i-m-4}$ and $\lambda^{a+2}r_i^{-4}$;
\item $\Sigma_0$ times a fixed compact-region density.
\end{enumerate}
The first three are transported as physical sources with factor $\lambda^{-a-2}$. Their cutoffs begin outside the fixed near collars. Actual sources adjacent to an entrance are retained in the separate near $C^{0,\alpha}$ slots, with norm divided by $\Sigma_i$; item 1 supplies their positive comparison potential. Choose the envelope constants uniformly on the fixed compact parameter set. After the domain pullbacks, $F_*$ and $E$ can then be fixed during parameter differentiation. Include the logarithmic Green remainders when $n=4$. In item 3 retain both the inner source, cut off at $2\rho_i/\lambda$, and the incoming shell.

These sources follow directly from \eqref{eq:outer-identity}, the bubble expansion
$$
B_i^{\rm phys}-\alpha_ir_i^{-m}
=O_{\scnorm}(\lambda^{a+1}r_i^{-m-1}+\lambda^{a+2}r_i^{-m-2}),
$$
and the second Green source identity
$$
P_0R_2=-(P_g-P_0-\dot P_0)A_0-(P_g-P_0)(A_1-A_0).
$$
The second Green identity contributes the source in item 2 and smooth incoming terms; its linear exterior contribution has already been cancelled. In fixed scaled cores, \eqref{eq:first-response-match} cancels the entire linear first source and its compensation. The remaining base source is $O(\tau_i^2+\tau_i\Theta_i+\Theta_i^2+\lambda^m)$; its field derivative is $O(\tau_i+\Theta_i+\Sigma_i)$. These bounds include cutoff commutators and the $B_{\rm ref}$ term.

We estimate the envelopes using the radial Newton formula. For radial $f$,
\begin{equation}\label{eq:radial-newton}
\mathcal Nf(r)=C_n\left(r^{-m}\int_0^rf(s)s^{m+1}\,ds+\int_r^\infty f(s)s\,ds\right).
\end{equation}
Define $H_\beta(R)=1$ for $\beta<m$, $1+\log R$ for $\beta=m$, and $R^{\beta-m}$ for $\beta>m$. The potential of item 2 has core size $C\tau_i^2H_{2d_i}(\rho_i/\lambda)$, outer monopole size $C\tau_i^2(\rho_i/\lambda)^{2d_i}$ and squared Dirichlet--Hardy norm at most $C\tau_i^4H_{4d_i}(\rho_i/\lambda)$. The potential of item 3 has core size $CK_i^{\rm in}$, monopole size $CK_i^{\rm in}(\rho_i/\lambda)^m$, and squared energy at most $C(K_i^{\rm in})^2(\rho_i/\lambda)^m$. These estimates follow by evaluating the two integrals in \eqref{eq:radial-newton}; at $\beta=m$, keep the factor $1+\log(R/r)$. The shell potential is comparable to a constant inside $R$ and to $R^mr^{-m}$ outside.

For $i\in\mathcal C$ take
$$
K_i^{\rm in}=C\{\lambda^{k_i}+\lambda^m+\mathbf1_{d_i=a}\lambda^mL_\lambda\},\qquad
k_i=m+1+\theta(d_i-m-1)>d_i.
$$
For $i\in\mathcal J$ take $K_i^{\rm in}=C\lambda^m$, with a factor $L_\lambda$ in dimension four. Indeed the same-core cost of item 4 is
$$
C\lambda^{m+1}\int_{\rho_i}^{r_0}r^{d_i-m-2}\,dr
+C\lambda^{m+2}\int_{\rho_i}^{r_0}r^{d_i-m-3}\,dr
+C\lambda^{m+2}\rho_i^{-2}.
$$
On a curvature end its first term is $O(\lambda^{k_i})$ and its second has an extra $\lambda/\rho_i$. On a mass end it is $o(\lambda^m)$; in dimension three the two terms are $O(\lambda^2L_\lambda+\lambda^3\rho_M^{-1})=o(\lambda)$. For $n\ge4$, the exponents $1-\gamma(m+1-q)$ and $2-\gamma(m+2-q)$ are positive.

The energy and relative estimates needed later are
\begin{equation}\label{eq:budget-energy}
\tau_i^4H_{4d_i}(\rho_i/\lambda)+(K_i^{\rm in})^2(\rho_i/\lambda)^m+\Sigma_i^2=o(\eps_i^2),\qquad
K_i^{\rm in}(\rho_i/\lambda)^m=o(1).
\end{equation}
For a noncritical curvature end, the energy of item 2 divided by $\lambda^{2d_i}$ has exponent $2d_i$ if $4d_i<m$, the same exponent with a logarithm at equality, and exponent $(m-2d_i)+(4d_i-m)\theta>0$ otherwise. At a critical end the corresponding exponent is $m\theta>0$, with an additional dividing logarithm. Also
$$
\lambda^{k_i}(\rho_i/\lambda)^m=\lambda^{1+(d_i-1)\theta},\qquad
\lambda^{2k_i}(\rho_i/\lambda)^m/\lambda^{2d_i}
=\lambda^{(m+2-2d_i)(1-\theta)}.
$$
At a mass end, $\tau_i^4H_{4q}(\rho_i/\lambda)=\lambda^m\rho_M^{4q-m}$ and $(K_i^{\rm in})^2(\rho_i/\lambda)^m=\lambda^m\rho_M^m$, up to the harmless finite logarithms. This proves \eqref{eq:budget-energy}.

The physical charges of items 1--3 are bounded, respectively, by $C\lambda^a\Sigma_i$, $C\lambda^a\rho_i^{2d_i}$ and $C\lambda^aK_i^{\rm in}\rho_i^m/\lambda^m$. After division by $\Sigma_0$, their compact-region effects are bounded by the respective quantities $\Sigma_i/\lambda^b$, $\rho_i^{2d_i}/\lambda^b$ and $K_i^{\rm in}\rho_i^m/\lambda^{m+b}$. Each tends to zero by the choices above. Conversely a compact-region source produces scaled core size $O(\lambda^{m+b})$. Cross-end costs have the same additional $\lambda^m$ factor. The exterior regular power source has squared energy bounded by
$$
C\lambda^{m+4}\begin{cases}
\rho_i^{m-4},&m<4,\\1+|\log\rho_i|,&m=4,\\1,&m>4,
\end{cases}
$$
which is $o(\lambda^m)$ in every dimension. For a curvature end, the first exterior source has squared energy at most $C\lambda^{m+2}\rho_i^{2d_i-m-2}$. Its normalized bound tends to zero because
$$
\frac{\lambda^{m+2}\rho_i^{2d_i-m-2}}{\lambda^{2d_i}}=(\lambda/\rho_i)^{m+2-2d_i}\longrightarrow0.
$$ The second has an extra factor $(\lambda/\rho_i)^2$. Formula \eqref{eq:radial-newton} proves these bounds also for $m=1,2$, with their respective power and logarithmic potentials.

\begin{proposition}\label{prop:projected}
For each sufficiently small $\lambda$ and the fixed compact parameter set, there is a $C^2$ family $U>0$, smooth on $M\setminus\Lambda$, with complete ends, satisfying
\begin{equation}\label{eq:projected}
\mathcal N_gU=-\sum_{i=1}^k\sum_{\alpha=0}^nc_{i\alpha}\Psi_{i\alpha}.
\end{equation}
In the fixed coordinate support of the compensations, $\Psi_{i\alpha}(x)=\lambda^{-a-2}\psi_{i\alpha}(x_i/\lambda)$. They are supported in fixed scaled compensation annuli; the coordinate and conformal extensions are the identity on these supports. The family uses the same references, sources and physical Cauchy matching defined above.
\end{proposition}

\begin{proof}
Let $F_*$ be the sum of the positive envelopes just listed, with the appropriate physical scalings, and $E=P_{g_0}^{-1}F_*$. Use the source space $\mathbb Y_E$ of Proposition~\ref{prop:global-inverse}, and measure each high entrance in $\Sigma_i^{-1}C^{2,\alpha}$. The computations above give $\lambda^aE(p_i)\le C\Sigma_i$ and $E/W_{\rm aff}^-=o(1)$ throughout the outer region. In the near collars $E$ is harmonic and comparable to $\lambda^{-a}\Sigma_i$, so the added annular density supplies the required local lower bound. The explicit source estimates give $\|\widehat f_{{\rm near},i}\|_{C^{0,\alpha}}\le C\Sigma_i$. These estimates verify the hypotheses of Proposition~\ref{prop:global-inverse}. The explicit reference and its source derivatives also bound the local source H\"older norms and the scaled H\"older norm of the reference linearized potential.

For a high entrance $\sigma$, first solve the complete tail source equation of Proposition~\ref{prop:end}, including the first response $Y[I]$. Let $(\mathfrak d(I,\sigma),\mathfrak n(I,\sigma))$ be its physical discrepancies from the reference. For an exterior source $f$, set $Z=\mathcal R_w^{\rm glob}(f;\mathfrak d,\mathfrak n)$ and let $S(f,\sigma)$ be its returned high trace. Define the exact exterior update
$$
F_{\rm ext}(f,\sigma)=\mathcal N_g(W_{\rm aff}^-+Z^-)+A^-Z^-+\sum_{i,\alpha}c_{{\rm geom},i\alpha}[I]\Psi_{i\alpha}.
$$
Split this residual, rather than the field, into the prescribed near and far source pieces. Solve
\begin{equation}\label{eq:common-fp}
\sigma=S(f,\sigma),\qquad f=F_{\rm ext}(f,\sigma).
\end{equation}
At a fixed point, $A^-Z=f+\sum_{i,\alpha}c_{{\rm corr},i\alpha}\Psi_{i\alpha}$, so \eqref{eq:projected} holds with $c_{i\alpha}=c_{{\rm geom},i\alpha}+c_{{\rm corr},i\alpha}$. The Cauchy equations make the physical traces and conormals equal.

To solve \eqref{eq:common-fp}, we estimate the four derivative blocks in the source and normalized entrance variables. For $W>0$, write $Q_W(z)=Q\{(W+z)^p-W^p-pW^{p-1}z\}$. In the positive relative ball, Taylor's formula gives
$$
|Q_W(z)-Q_W(z')|\le C\eta_\lambda W^{p-1}|z-z'|,\qquad\eta_\lambda\to0.
$$
Its H\"older estimate follows by the same integral Taylor formula and the actual scaled field estimates. The three-kernel bound and the inverse give a small source norm. The known potential difference between $W_{\rm aff}^-$ and $W_0$ also tends to zero at fixed $r_0$: on the long scaled annuli its coefficients are $O(\lambda^{d_i}r^{d_i-4})$ and $O(\lambda^mr^{m-4})$. Their Newton integral norms tend to zero, with logarithmic factors when an integration exponent vanishes. The near part is $O(C(t_0)(\tau_i+\Theta_i+\eps_i^2))$. Thus shrinking $\lambda$ makes these operator errors small relative to the fixed inverse bounds.

For fixed finite parameters, \eqref{eq:first-response-match} implies the exact identity
$$
A^-D_I^jY^-=\sum_{i,\alpha}(D_I^jc_{{\rm geom},i\alpha})\Psi_{i\alpha},\qquad j=1,2.
$$
In $F_{\rm ext}$, the compensation derivative cancels the linear derivative of the geometric response. The remaining source derivatives are bounded; the derivative of $I$ with respect to the normalized high entrance has size $C\max_i\Sigma_i/\eps_i=o(1)$. The discrepancy derivatives in \eqref{eq:raw-entrance} contribute $o(1)$ as well. The variation of the global inverse with $I$ changes only the high boundary response and is $o(1)$ on fixed-core and exterior norms, by the boundary resolvent. Consequently, with $\widehat\sigma_i=\sigma_i/\Sigma_i$, the four blocks satisfy
\begin{align*}
\|D_fF_{\rm ext}\|&\le C\eta,&
\|D_{\widehat\sigma}F_{\rm ext}\|&\le o(1)+C\eta,\\
\|D_f(S/\Sigma)\|&\le C,&
\|D_{\widehat\sigma}(S/\Sigma)\|&=o(1),
\end{align*}
where the source derivatives use $\mathbb Y_E$, and $\widehat\sigma_i$ uses the fixed $C^{2,\alpha}$ norm. Proposition~\ref{prop:global-inverse} supplies the uniform source-to-trace bound: both the near inputs and the returned trace were divided by the same $\Sigma_i$. The physical discrepancy terms contribute $O(\eps_i\Sigma_i)/\Sigma_i=O(\eps_i)$ to the entrance derivative, by \eqref{eq:raw-entrance}. Their second normalized derivatives are bounded by $C\Sigma_i$, and the source derivatives of the tail fixed point carry $\Sigma_i/\eps_i$ or its square. All these estimates are taken with the normalization of the corresponding end. The quantity $\eta$ can be made arbitrarily small with all fixed data held fixed. First solve the entrance equation by contraction for each $f$, and substitute its solution in the source equation. The resulting derivative is less than $1/2$ for small $\lambda$.

The zero-order residual is bounded in these normalized spaces, by \eqref{eq:budget-energy} and the core estimates. Choose the exterior ball large enough to contain it, then choose the high entrance ball large enough for its inverse image. These two fixed radii are selected after $t_0$ and before $\lambda$; the factors $\Sigma_i/\eps_i\to0$ ensure that they do not enlarge the earlier basic tail source ball. The two maps preserve these balls, and the substituted source map is a contraction. Difference quotients and the differentiated fixed-point equations in the fixed domains, using the weaker tail weights, give $C^2$ dependence.

The exterior is positive because $|Z^-|\le o(1)W_{\rm aff}^-$ and $W_{\rm aff}^->0$. The latter follows from $U_0>0$ and $J^-/U_0=o(1)$, by the explicit Green and bubble estimates. The interior is at least $w/2$. Physical value and conormal matching eliminate interface distributions. Local elliptic regularity gives smoothness, and the lower end bound is $U\ge C_i\eps_ir_i^{-a}$ in physical coordinates. It already implies completeness at each puncture for every fixed member of the family.
\end{proof}

\section{The homogeneous geometric coefficients}\label{sec:coefficients}

We first compute the quadratic coefficients for a single homogeneous metric term. The subcritical coefficient below is the negative of the form in \cite[Propositions~6.7 and~6.9]{HigherI}, divided by $\omega$; for the critical quantity of Lemma~\ref{lem:critical}, $-a\Gamma(H)$ is the degree-$a$ diagonal logarithmic coefficient of the Pohozaev quadratic form in \cite[Appendix~A]{KMS}. We give the calculations in the normalization of our action. Section~\ref{sec:action} transfers them to the family of Proposition~\ref{prop:projected}, with the physical boundary terms included.

\begin{lemma}\label{lem:subcritical}
Let $H_d$ be a nonzero homogeneous polynomial symmetric tensor on $\R^n$, of degree $2\le d<a$, with $H_dx=0$ and $\tr H_d=0$. The complete-bubble second variation, after its scalar first response is solved, has coefficient $-cF_d(\zeta)/(2\omega)$, where
$$
F_d(\zeta)>0,\qquad \deg F_d\le2d-4,\qquad
F_d(\zeta)/(1+|\zeta|^2)^{2d}\longrightarrow0.
$$
Here $\zeta$ is the translated bubble center at unit width, and
$$
F_d(\zeta)=\frac14\int_{\Sph^n}(|\nabla T_d|^2+2|T_d|^2),
$$
where $T_d$ is the transverse trace-free part of the total compactified metric variation.
\end{lemma}

\begin{proof}
The compactified tensor belongs to $H^1$, since its tail energy is a constant multiple of $\int_1^\infty r^{2d-n+1}\,dr$ and $d<a$. On the unit round sphere $(\Sph^n,\gamma)$, the transverse trace-free projection is bounded in $H^1$: for an $H^1$ tensor $h$, use the conformal Killing operator $\mathcal KX=\mathcal L_X\gamma-(2/n)(\diver X)\gamma$ and solve
$$
\mathcal K^*\mathcal KX=\mathcal K^*\bigl(h-\tfrac1n(\tr_\gamma h)\gamma\bigr),\qquad X\perp\ker\mathcal K.
$$
The right-hand side is in $L^2$ and orthogonal to the kernel, so elliptic estimates give $\|X\|_{H^2}\le C\|h\|_{H^1}$. Subtracting $\mathcal KX$ from the trace-free part of $h$ gives its transverse trace-free part with the same $H^1$ bound. Smooth cutoff approximation therefore also gives convergence of these parts in $H^1$, and the sphere Hessian below extends to the compactified variation. On the unit round sphere the scalar curvature linearization has formal adjoint
$$
DR^*\varphi=\nabla^2\varphi-(\Delta\varphi)\gamma-\varphi\Ric.
$$
For every first harmonic this is zero, since $\nabla^2\varphi=-\varphi\gamma$, $\Delta\varphi=-n\varphi$ and $\Ric=(n-1)\gamma$. Consequently the full metric first source is orthogonal to the bubble Jacobi fields. The scalar first response is solvable by \eqref{eq:bubble-inverse}; its compensation is zero. The total compactified metric variation $h$ satisfies $DR_\gamma h=0$. Write $h=\mathcal L_X\gamma+f\gamma+T_d$, where the transverse trace-free tensor $T_d$ is orthogonal to the sum of the Lie and scalar variations; see \cite[Lemma~2.4.1 and Remark~2.4.2]{Kroncke}. Since the Lie and transverse trace-free terms have zero scalar curvature variation,
$$
0=DR_\gamma h=-(n-1)(\Delta+n)f.
$$
Thus $f$ is a first harmonic and $\mathcal L_{\nabla f}\gamma=2\nabla^2f=-2f\gamma$, so the scalar term is absorbed into the Lie derivative by replacing $X$ with $X-\tfrac12\nabla f$.

Consider the stationary functional $\int(R-(n-1)(n-2))\,dV$. On transverse trace-free tensors its Hessian on the unit sphere is
$$
-\frac12\int_{\Sph^n}(|\nabla T|^2+2|T|^2).
$$
This is the Einstein action second variation with curvature convention $\mathring RT=-T$ on trace-free tensors; see \cite[Remark 2.2.6 and (2.6)]{Kroncke}. Diffeomorphism invariance at the stationary metric makes every Hessian pairing with a Lie derivative vanish. The function action is $c/(2\omega)$ times this metric action. Including the Taylor factor $1/2$ gives exactly $-cF_d/(2\omega)$.

If $T_d$ were zero, the original polynomial tensor would be a pure conformal coordinate variation. The decomposition is smooth near the finite origin by local elliptic regularity. Taking its homogeneous Taylor terms would give $H_d=\mathcal L_X\delta+2f\delta$, with $X,f$ homogeneous of degrees $d+1,d$. The radial and trace conditions imply
$$
X=\frac{r^2\nabla f-2dfx}{d(d+1)},\qquad
r^2\Delta f+(n-2)d(d-1)f=0.
$$
On $r^{2j}$ times a degree-$d-2j$ harmonic polynomial, $r^2\Delta$ has the nonnegative eigenvalue used in Section~2. The second term is strictly positive; hence $f=X=H_d=0$, a contradiction. The same argument applies after translation, by translating a putative pure gauge representation back before taking its homogeneous degree-$d$ term.

In $H_d(y+\zeta)$, the terms of degrees zero and one in $y$ have zero linearized curvature and are pure coordinate variations. They contribute no transverse trace-free tensor. The remaining coefficients have polynomial degree at most $d-2$ in $\zeta$. The transverse trace-free projection is linear, and its energy is quadratic. Thus $F_d$ is a polynomial of degree at most $2d-4$. The preceding argument gives strict positivity, and the degree bound gives the stated decay after division by $(1+|\zeta|^2)^{2d}$.
\end{proof}

The borderline case requires a separate computation, because its complete-bubble energy has logarithmic growth.

\begin{lemma}\label{lem:critical}
Let $n=2a+2$, where $a\ge2$ is an integer. Let $B$ be a nonzero homogeneous polynomial symmetric tensor of degree $a$ with $Bx=0$ and $\tr B=0$, and let $H$ be its tangential restriction to the unit sphere $\Sph^{2a+1}$. Write
$$
R_{dr^2+r^2\gamma\exp(r^aH)}=r^{a-2}L+r^{2a-2}K+O(r^{3a-2}),\quad
A=c(\Delta_S-a^2)^{-1}L.
$$
Then $\Gamma(H)=\int_{\Sph^{2a+1}}(K+LA)<0$. The coefficient of the quadratic logarithmic term in the physical function action with monopole $\alpha r^{-2a}$ is $c\alpha^2\Gamma(H)/(2\omega)$.
\end{lemma}

\begin{proof}
Put $s=2a+1$. The shape endomorphism of the spheres is exactly $r^{-1}I+(a/2)r^{a-1}H$. Its trace is $s/r$ and its squared trace norm is $s/r^2+(a^2/4)r^{2a-2}|H|^2$. The Gauss formula gives
$$
R=r^{-2}\{R_{\gamma\exp(r^aH)}-s(s-1)\}
-\frac{a^2}{4}r^{2a-2}|H|^2.
$$
The first variation is $L=\diver_S\diver_SH$. The stationary sphere action $\int(R-(s-1)(s-2))\,dV$ has, on trace-free tensors, Hessian
$$
-\frac12\{\|\nabla H\|_2^2+2\|H\|_2^2-2\|\diver H\|_2^2\}.
$$
This follows by integration of the Ricci variation
$$
\int\langle\Ric'H,H\rangle
=\frac12\{\|\nabla H\|_2^2+2s\|H\|_2^2-2\|\diver H\|_2^2\}
$$
and the variation of the Einstein gradient. The path $\gamma\exp(tH)$ has constant volume density; its acceleration contributes zero because the action is stationary. The Taylor coefficient is half the Hessian. Hence
\begin{equation}\label{eq:critical-K}
\int K=-\frac14\{\|\nabla H\|_2^2+(a^2+2)\|H\|_2^2-2\|\diver H\|_2^2\}.
\end{equation}

To determine the scalar part of this quadratic form, use tangentiality and trace freeness to obtain
$$
\partial_i\partial_jB_{ij}=r^{a-2}\diver_S\diver_SH.
$$
The left side is a homogeneous polynomial of degree $a-2$. Also $\int L=0$ and $\int L\theta_\nu=\int H:\nabla^2\theta_\nu=0$. Thus only
$$
\mathcal I_a=\{\ell:2\le\ell\le a-2,\ \ell\equiv a\pmod2\},\qquad
\lambda_\ell=\ell(\ell+2a)
$$
can occur. For a scalar function $f$, write $H_f=\nabla^2f-(\Delta f/s)\gamma$. Decompose
$$
H=H_T+\mathcal L_V\gamma+\sum_{\ell\in\mathcal I_a}H_{f_\ell},\qquad
\diver H_T=0,\quad\diver V=0.
$$
This orthogonal decomposition is obtained by solving $((s-1)/s)\Delta(\Delta+s)f=L$ on $\ell\ge2$ and then $(\Delta+s-1)V=\diver(H-\sum_{\ell\in\mathcal I_a}H_{f_\ell})$ on coclosed vector fields. The latter right side is perpendicular to Killing fields; its kernel is exactly those fields because $-\int\langle(\Delta+s-1)V,V\rangle=\frac12\|\mathcal L_V\gamma\|_2^2$. The two elliptic equations therefore give the asserted orthogonal decomposition.

For $-\Delta f=\lambda f$, the integrated Bochner formula gives
$$
\|H_f\|_2^2=\frac{s-1}{s}\lambda(\lambda-s)\|f\|_2^2,\qquad
\diver H_f=-\frac{s-1}{s}(\lambda-s)\nabla f.
$$
Commuting derivatives on the unit sphere gives $\nabla^*\nabla H_f=(\lambda-2s)H_f$. Substitution in \eqref{eq:critical-K}, followed by $A=-cL/(a^2+\lambda)$ on this mode, gives
\begin{align*}
\frac{\Gamma(H_f)}{\|H_f\|_2^2}
&=-\frac14\left[a^2-\frac{s-2}{s}\lambda
+\frac{4a^2\lambda(\lambda-s)}{s^2(a^2+\lambda)}\right]\\
&=-\frac{(a^2s-\lambda)^2}{4s^2(a^2+\lambda)}.
\end{align*}
In the last equality $4a^2-s(s-2)=1$ is used. On $H_T$, the divergence terms vanish. On the vector Lie derivative the sphere Hessian vanishes, leaving the radial shape term. Orthogonality removes their cross terms; for the scalar blocks this also follows from the rough Laplacian eigen-equation. Therefore
\begin{equation}\label{eq:critical-square}
\Gamma(H)=-\frac14\left[\|\nabla H_T\|_2^2+(a^2+2)\|H_T\|_2^2
+a^2\|\mathcal L_V\gamma\|_2^2
+\sum_{\ell\in\mathcal I_a}\frac{(a^2s-\lambda_\ell)^2}{s^2(a^2+\lambda_\ell)}\|H_{f_\ell}\|_2^2\right].
\end{equation}
For every allowed scalar mode, $\lambda_\ell\le(a-2)(3a-2)<a^2(2a+1)$, since the difference is $2a^2(a-1)+8a-4>0$. Every nonzero block is strictly negative. Tangentiality and homogeneity imply $H\ne0$, which proves the strict sign.

For the logarithmic normalization, the first scalar response to $u_0=\alpha r^{-2a}$ is $w_1=\alpha r^{-a}A$. All relevant radial integrals are $\int dr/r$. Their coefficient before division by $\omega$ is
$$
\alpha^2\left[\frac12\int(|\nabla A|^2+a^2A^2)+c\int LA+\frac c2\int K\right]
=\frac{c\alpha^2}{2}\Gamma(H),
$$
where $(\Delta-a^2)A=cL$ was used. This proves the stated coefficient and its sign.
\end{proof}

\section{The action of the projected family}\label{sec:action}

For a finite physical domain $D$, use the normalized action
$$
\mathcal A_g(u;D)=\frac1\omega\int_D
\left(\frac12|\nabla u|_g^2+\frac c2R_gu^2-\frac Q{p+1}u^{p+1}\right)d\mu_g.
$$
Its interior first variation is $-\mathcal N_gu$. For a singular end, renormalization must use its source-dependent terminal phase to cancel the terminal boundary variation.

\subsection{The end counterterm and the finite interfaces}

Fix $\lambda$ and the $\eps_i$, set the terminal energy equal to $1$, and use the entrance parameters in \eqref{eq:parameter-change}. In this subsection, $O_{C^2}$ and $o_{C^2}$ are joint operator-norm estimates through order two in $I\in\mathcal Y_{x_4,y_4}$, the entrance parameters and the normalized additional high trace $\sigma/\Sigma$, on smaller open source balls and compact entrance sets. Pull all domains back as in Proposition~\ref{prop:global-inverse} and retain the full source dependence of the first response from Section~\ref{sec:projected}. We establish these estimates before substituting the common nonlinear fixed point. The positive weights $x_4$ and $y_4-1$ give the integrable tail bounds used below.

Let $h_\eps(e)=-a^2\eps^2e/2$ and define, on the real phase cover,
$$
\mathscr S_\eps(e,s)=\int_0^s\{(v_{\eps,e}'(\sigma))^2-h_\eps(e)\}\,d\sigma.
$$
At zero neck set $\mathscr S_0(s)=\int_0^s(v_0'(\sigma))^2\,d\sigma$. Formula \eqref{eq:wronskian} gives $\partial_e\mathscr S_\eps=v_sv_e$ and $\partial_s\mathscr S_\eps=v_s^2-h_\eps$. On the actual radial state this implies the exact identity
$$
\frac12w'^2+\frac{a^2}{2}(w^2-w^{p+1})
=\frac d{dt}\mathscr S_\eps(e(t),t+\phi(t))+h_\eps(e(t))\phi'(t).
$$
Subtracting the actual cylindrical terminal counterterm $\mathscr S_\eps(1,T+\phi_\infty)+h_\eps(1)\phi_\infty$ and taking the limit gives exactly
$$
-\mathscr S_\eps(e(t_0),t_0+\phi_0)-h_\eps(1)\phi_0
+\int_{t_0}^{\infty}\{h_\eps(e(t))-h_\eps(1)\}\phi'(t)\,dt.
$$
Here $\int h_\eps(1)\phi'=h_\eps(1)(\phi_\infty-\phi_0)$ cancels the terminal phase term, including all its source derivatives. Lemma~\ref{lem:fixed-entry} fixes $\phi_0$ while allowing $\phi_\infty$ to vary. The remaining integral is
$$
-\frac{a^2\eps^2}{2}\int_{t_0}^\infty(e-e_\infty)\phi'\,dt=o_{C^2}(\eps^2).
$$
For this little-oh estimate, keep $t_0$ fixed. Then \eqref{eq:entry-cauchy} and $v_e=O(\eps^2)$ give $\phi'=O_T(\eps^2)$ on every fixed cylinder, with the joint derivatives just specified. The weaker positive weights from Proposition~\ref{prop:end} give an integrable polynomial-times-exponential bound on the entire tail for the integrand divided by $\eps^2$ and for each of its first two derivatives. Splitting at a fixed large cylinder proves convergence to zero in the joint operator norm, even when the sources vary with $\eps$.

Let $\mathcal I_\Gamma(w)$ denote the angular average of the action density at time $t$. The metric $\gamma_{\mathbf0}$ is the standard product cylinder obtained by setting $\mathbf q=0$ in \eqref{eq:moving-cylinder}. The moving coordinates contribute another term. Substitution of \eqref{eq:moving-cylinder} in the energy density gives
\begin{equation}\label{eq:moving-action}
\mathcal I_{\gamma_{\mathbf q}}(w)-\mathcal I_{\gamma_{\mathbf0}}(w)
=\frac12J^2\mathcal U(\mathbf q)+\frac a2\frac d{dt}\{w^2\mathcal U(\mathbf q)\}.
\end{equation}
For this computation the volume density is $1-\mathbf q\cdot\theta$, its mean is one, and integration of the angular divergence $\diver_S(\nabla_S(\mathbf q\cdot\theta)/(1-\mathbf q\cdot\theta))$ removes the remaining angular derivative terms. Write $Z_q=wJ\mathbf q/\eps^2$, which has a positive exponential rate by \eqref{eq:low-estimates} and the weaker parameter weights. Since $\mathcal U(0)=D\mathcal U(0)=0$, Taylor's integral formula gives exactly
$$
\eps^{-2}J^2\mathcal U(\mathbf q)=\frac{\eps^2}{w^2}\int_0^1(1-s)D^2\mathcal U(s\mathbf q)[Z_q,Z_q]\,ds.
$$
This is bounded by an integrable exponential times a polynomial, and tends to zero on every fixed cylinder, with two parameter derivatives. Its integral is therefore $o_{C^2}(\eps^2)$. We keep the entrance term in \eqref{eq:moving-action} until combining it with the physical conformal boundary term.

If $u\circ K=\Omega^{-a}U$ and $K^*g=\Omega^2\Gamma$, expansion of the conformal energy gives
\begin{equation}\label{eq:conformal-action}
\mathcal A_g(u;K(D))=\mathcal A_\Gamma(U;D)
-\frac a{2\omega}\int_{\partial D}U^2\partial_{\nu_\Gamma}\log\Omega\,d\sigma_\Gamma.
\end{equation}
The $|d\log\Omega|^2$ coefficient cancels because $a^2=c(n-1)(n-2)$; the remaining mixed terms form the displayed divergence. At infinity, $\partial_t\log\Omega=-1+O(r)$ and $U-v_{\eps,1}(t+\phi_\infty)\to0$. The physical counterterm at $T$ is therefore
\begin{equation}\label{eq:counterterm}
\mathscr C_{i}(T)=\mathscr S_{\eps_i}(1,T+\phi_{\infty,i})
-\frac{a^2\eps_i^2}{2}\phi_{\infty,i}
+\frac a2v_{\eps_i,1}(T+\phi_{\infty,i})^2.
\end{equation}
Its phase derivative is $(v'+av)v'$, exactly the physical terminal flux. In source and entrance derivatives, $\phi_{\infty,i}$ is the function supplied by Lemma~\ref{lem:fixed-entry}, and its derivatives are included in this counterterm. Define $\mathcal A_{\rm ren}$ by subtracting \eqref{eq:counterterm} at all ends before taking $T\to\infty$. The preceding positive weights justify this limit and its two parameter derivatives. Moving from coordinate spheres to fixed physical spheres changes the logarithmic radius by $O(r)$ and hence does not change the limit.

We verify the entrance cancellation explicitly. Put $\mathfrak l_\Gamma=\rho_\Gamma\Gamma^{tj}\partial_j\log\Omega$. For the flat static and moving special conformal charts,
$$
\mathfrak l_{\gamma_{\mathbf q}}-\mathfrak l_{\gamma_{\mathbf0}}
=\frac{\mathbf q\cdot\theta}{1-\mathbf q\cdot\theta}.
$$
To check the identity, substitute $\Omega=r/(1-2rb\cdot\theta+r^2|b|^2)$ into $(1-\mathbf q\cdot\theta)^{-1}(\partial_t\log\Omega+\nabla_S(\mathbf q\cdot\theta)\cdot\nabla_S\log\Omega)$ and subtract the static expression. The numerator reduces to the displayed quotient. In the true finite-normal charts the additional difference is $O(\tau|\mathbf q|)$. Thus the entrance term in \eqref{eq:conformal-action}, summed over the two pieces, is $(a/2)w(t_0)^2\mathcal U(\mathbf q(t_0))+o_{C^2}(\eps^2)$. It cancels the lower boundary in \eqref{eq:moving-action}.

Let $B(t)=v_0(t+\phi_0)$, with the actual entrance phase fixed as in Lemma~\ref{lem:fixed-entry}, and put $e_0=e(t_0)$. Smooth dependence of the peak initial value on $\eps^2e_0$ gives
$$
\mathscr S_\eps(e_0,t_0+\phi_0)-\mathscr S_0(t_0+\phi_0)
=B'(t_0)(w(t_0)-B(t_0))+O_{C^2}(\eps^4).
$$
The actual core first variation has the opposite entrance term, because its value discrepancy is $w(t_0)-B(t_0)$ plus a mean-zero high trace. The other boundary is kept and cancels with the actual exterior first variation. Consequently the end contributes exactly
\begin{equation}\label{eq:end-phase-action}
\frac{a^2\eps_i^2}{2}\phi_i
\end{equation}
to the parameter-dependent leading action, in addition to the geometric quadratic terms. Here $\phi_i$ is the entrance phase in \eqref{eq:parameter-change}; the dependence on the terminal phase has cancelled in the preceding identities.

The finite-normal remainder contributes both volume and boundary terms. For an arbitrary sphere metric $G(t)$ let $\rho=dV_G/dV_\gamma$, $f=\langle\rho\rangle-1$, $\mathsf K=G^{-1}G'/2$ and $\mathsf H=\tr\mathsf K$. The Gauss identity is $R_{dt^2+G}=R_G-|\mathsf K|^2-\mathsf H^2-2\mathsf H'$, and $\rho'=\rho\mathsf H$. Define the stationary sphere action $\mathcal S(G)=\langle(R_G-(n-2)(n-3))\rho\rangle$. Two integrations by parts in the volume term give the exact formula
\begin{align*}
\mathcal A_{dt^2+G}(w)-\mathcal A_{\gamma_{\mathbf0}}(w)
={}&\frac c2\int w^2\{\mathcal S(G)-\mathcal S(\gamma)+\langle\rho(\mathsf H^2-|\mathsf K|^2)\rangle\}\,dt\\
&+\int f\left(\frac{E_w}{2(n-1)}-2cwF\right)dt
-c[w^2f'-2ww'f].
\end{align*}
For the finite-normal sphere metric $G=\gamma e^{B_{\tau,b}}+E_{\tau,b}$, $f,f'=O(\tau r^N)$ with the stated parameter factors, so the volume integral is $O_{C^2}(\tau\eps^2)=o(\eps^2)$. The boundary is not discarded. The full tensor $\gamma B_{\tau,b}+E_{\tau,b}$ is included in the quadratic variation of the first integral, including its mixed term and the square of $E_{\tau,b}$. The mixed moving-metric density difference is $O_{C^2}(\eps^2\tau\eps^{-4/n})=o(\eps^2)$: factor each path rate as $\eps^2/(wJ)$ times its controlled field before estimating the density. It uses no derivative of $F$.

We next estimate the nonlinear field remainder in energy. Write $h=\eps^2w^{-1}\eta$. Then
$$
\frac h\eps=\frac\eps w\eta,\qquad
\frac {h_t}\eps=\frac\eps w(\eta_t-(w'/w)\eta).
$$
The same formula after two parameter derivatives has only relative derivatives of $w$, bounded by time polynomials. Since $\eps/w\le C$, positive decay of $\eta$ gives an integrable bound, and on fixed cylinders $\eps/w\to0$. Thus $h=o_{C^2}(\eps)$ in a slightly weighted $H^1$ norm. On curvature ends the first response has norm $O_{C^2}(\tau)$, and its quadratic term is retained in the geometric functional. On a mass end its norm is $O_{C^2}(\tau_i+\Theta_i)=o(\eps_i)$, by the estimates in Section~\ref{sec:projected}, so its quadratic term is $o_{C^2}(\eps_i^2)$. The additional homogeneous entrance has norm $O_{C^2}(\Sigma)=o(\eps)$. Since $\tau=O(\eps)$, its quadratic terms and its pairings with the first response, all terms containing $h$, and the higher-order remainders are $o_{C^2}(\eps^2)$. Linear $h$ has zero entrance value; the linear high-entrance boundary equals
$$
-w'(t_0)\left\langle\sigma[\rho_\Gamma\Gamma^{tt}-1]\right\rangle
=O_{C^2}((\tau+\eps^2)\Sigma)=o(\eps^2),
$$
using the exact mean-zero high trace. Noninteger powers are controlled by the positive relative Taylor integral, as in Proposition~\ref{prop:end}.

It remains to combine the finite interfaces. The field boundary $YY_t/2$, the mixed metric-field conormal term and the finite-volume boundary above cancel the corresponding core terms, since they arise from the same finite-domain second variation with opposite normals. For the latter, the mismatch is bounded by $C\tau\eps^2$ because $w-B,w'-B'=O_{C^2}(\eps^2)$. For the field boundary, the complete matched response has the same value and model derivative; its moving conormal difference is $O(\mathbf q)$ and is already included in \eqref{eq:conformal-action}. This establishes the boundary matching required to use the full-bubble quadratic functional.

\subsection{Transfer of the curvature coefficients}

\begin{proposition}\label{prop:curvature-rows}
For the family of Proposition~\ref{prop:projected} and $i\in\mathcal C$, define
\begin{equation}\label{eq:curvature-action}
\mathcal E_i=\frac{a^2}{2}\eps_i^2\{s_i+\log(1+|\zeta_i|^2)\}
-\eps_i^2e^{2d_is_i}K_i(\zeta_i),
\end{equation}
where $K_i=cF_{d_i}/(2\omega)$ on $\mathcal C_<$ and $K_i=-c2^m\Gamma_i/(2\omega)$ on $\mathcal C_=$. Here $\Gamma_i$ is the value of $\Gamma$ from Lemma~\ref{lem:critical} on the tangential restriction of $H_{d_i}$ at $p_i$. For any parameter direction $\mu$ and any direction $\mu_i$ belonging to end $i$,
\begin{equation}\label{eq:curvature-row-errors}
D_{\mu_i}\mathcal A_{\rm ren}=D_{\mu_i}\mathcal E_i+o(\eps_i^2),\qquad
D_\mu D_{\mu_i}\mathcal A_{\rm ren}=D_\mu D_{\mu_i}\mathcal E_i+o(\eps_i^2).
\end{equation}
The estimates hold uniformly on the fixed parameter set.
\end{proposition}

\begin{proof}
For the action comparison, fill each end across its fixed scaled entrance with the same static bubble and a bounded smooth collar extension of the difference. This auxiliary field retains the physical traces and the complete first response. On the $i$th fixed physical ball the squared energy of its own nonlinear correction is bounded by
$$
C\{\Sigma_i^2+\tau_i^4H_{4d_i}(\rho_i/\lambda)+(K_i^{\rm in})^2(\rho_i/\lambda)^m
+\mathcal Q_{{\rm ext},i}+\Sigma_0^2+\eps_i^4\}=o(\eps_i^2).
$$
Here $K_i^{\rm in}$ is the incoming amplitude of Section~\ref{sec:projected}, and $\mathcal Q_{{\rm ext},i}$ is bounded by its explicit exterior energy estimates. Near-source and Cauchy lifts contribute the $\Sigma_i^2$ and $\eps_i^4$ terms by Proposition~\ref{prop:global-inverse}. Each other-end source has physical charge $\lambda^a\delta_j$ with $\delta_j\to0$. Its energy in this ball is $O(\lambda^m\delta_j^2)=o(\eps_i^2)$, rather than its possibly larger energy at its own end. Estimate \eqref{eq:inverse-energy} is applied separately to the far volume-source envelopes, while the near and boundary columns retain their individual energy bounds.

Choose a smooth physical cutoff $0\le\chi_{\rm geom}\le1$, equal to one for $|x|\le r_0$ and zero for $|x|\ge2r_0$, independently of the finite bubble parameters. Compare the complete first response with the Euclidean response to the full cut-off tensor $\chi_{\rm geom}(\lambda y)\log g(\lambda y)$, in the same zero value-gradient convention. The static coordinate variation contributes the term $-\mathscr P_0^{-1}\dot{\mathscr P}_0B$ to the physical first response; it is included by the exact $B_{\rm ref}$ in Section~\ref{sec:projected}. Thus the difference between the finite coordinate change and its linearization is $O(\tau_i^2)$ in fixed cores. Keep the actual forced tail while comparing sources and solve the field difference and its compensation together with the augmented inverse. Only after this comparison pass to the full-bubble tail. The boundary resolvent gives strong inner $H^1$ convergence after division by $\tau_i$. Derivatives of $\chi_{\rm geom}$ are confined to a fixed physical outer annulus; their response has squared energy $O_{C^2}(\lambda^m)$ and is included in the bounds below.

In the long scaled collar, the remaining source differences are bounded by $\tau_i^2r^{2d_i-m-2}$, higher jets, and smooth incoming fields. Their squared energy bounds are
$$
C\tau_i^4H_{4d_i}(r_0/\lambda),\qquad
C\lambda^{2d_i+2}H_{2d_i+2}(r_0/\lambda),\qquad C\lambda^m.
$$
Each is $o(\eps_i^2)$. For $2d_i<m$, a growing first term is $O(\lambda^m)$ and its ratio to $\lambda^{2d_i}$ tends to zero; at $2d_i=m$ all three are $O(\lambda^m)$ and are divided by $\lambda^m\log(1/\lambda)$. The local cubic metric and field terms are bounded by $C\lambda^{3d_i}H_{3d_i}(r_0/\lambda)=o(\eps_i^2)$. Finite $E_N$ and coordinate factors satisfy the same higher-order estimates.

In the following Euclidean identities, the scaled full bubble and its linearized operator are
$$
B(y)=\left(\frac{2e^{s_i}}{e^{2s_i}+|y-e^{s_i}\zeta_i|^2}\right)^a,\qquad
A_B=-\Delta-QpB^{p-1}.
$$
Before applying these energy bounds, cancel the linear terms and the terms pairing the first response with the difference. For a filled difference $D$ in the permitted decay class, the exact full-bubble equation gives
$$
\int(\nabla B\cdot\nabla D-QB^pD)=0.
$$
For a fixed smooth compactly supported symmetric metric variation $h$ and scalar fields in this decay class, define
\begin{align*}
\mathcal Q_B(S,D)&=\frac1\omega\int_{\R^n}
\{\nabla S\cdot\nabla D-QpB^{p-1}SD\}\,dy,\\
\mathcal J_0(S)&=\tfrac12\mathcal Q_B(S,S)
+\frac1\omega\int_{\R^n}S
\left.\frac d{dt}\right|_{t=0}(P_{\exp(th)}B)\,dy.
\end{align*}
The part of the quadratic action depending on the scalar response is $\mathcal J_0$; the terms depending only on $h$ cancel in differences. Apply these definitions with $h(y)=\chi_{\rm geom}(\lambda y)\log g(\lambda y)$. Its Euclidean scalar first response $S_0$, in the zero value-gradient normalization, satisfies
$$
A_BS_0=-\left.\frac d{dt}\right|_{t=0}(P_{\exp(th)}B).
$$
The scalar-curvature adjoint calculation in the proof of Lemma~\ref{lem:subcritical} makes this source orthogonal to the Jacobi fields, so its compensation coefficients vanish. Testing this equation against $D$ gives
$$
\mathcal Q_B(S_0,D)+\frac1\omega\int_{\R^n}D
\left.\frac d{dt}\right|_{t=0}(P_{\exp(th)}B)\,dy=0.
$$
Expansion of the quadratic functional therefore yields
$$
\mathcal J_0(S_0+D)-\mathcal J_0(S_0)=\tfrac12\mathcal Q_B(D,D).
$$
Apply the identities first with finite cutoffs, then pass to the limit using the stated decay. In the critical case they remove the possible cross term with $\|S_0\|_{H^1}=O(\tau_i\sqrt{\log(1/\lambda)})$. In particular a constant of size $\lambda^m\log(1/\lambda)$ in the second Green remainder does not leave a spurious linear contribution. The preceding boundary calculations cancel the remaining entrance terms on the same field. Lemmas~\ref{lem:subcritical} and \ref{lem:critical} now give \eqref{eq:curvature-action}. The critical logarithm is integrated to the fixed physical radius $r_0$, so its upper endpoint is $r_0/\lambda$, not $\rho_i/\lambda$. Changing $\log(r_0/(\lambda e^{s_i}))$ to $\log(1/\lambda)$ costs $O(\lambda^m)=o(\eps_i^2)$.

To control derivatives belonging to one end, we also need estimates for their effect at every other end. If $j\ne i$, differentiating the actual common source equation \eqref{eq:common-fp} with respect to $\mu_i$ gives
$$
\eps_j\|D_{\mu_i}I_j\|+\|D_{\mu_i}\sigma_j\|\le C\lambda^m.
$$
The source norm here is $\mathcal Y_{x_4,y_4}$, and $\sigma_j$ uses the unnormalized $C^{2,\alpha}$ norm. The same bound holds after one arbitrary finite-parameter derivative. Indeed, the tail block has its previously fixed contraction constant and the returned entrance block has derivative $o(1)$ in the product norms of Proposition~\ref{prop:global-inverse}; their direct cross-end input is $O(\lambda^m)$. Eliminate the entrance block and then invert the tail block. In the second derivative, a raw $\eps_j^{-2}$ entrance derivative is multiplied by one cross-end trace $O(\lambda^m)$ and one ordinary trace derivative $O(\Sigma_j)$. In the source norm displayed above its contribution is
$$
\eps_j\,\eps_j^{-2}\lambda^m\Sigma_j
=\lambda^m\frac{\Sigma_j}{\eps_j}=o(\lambda^m).
$$
The remaining differentiated blocks are bounded in the same product norms and have the same cross-end factor. The fixed-core cross-end field derivative is therefore $O(\lambda^m)$.

On another fixed physical ball, use the first variation of the action. The counterterm \eqref{eq:counterterm} cancels the terminal flux, leaving the fixed compensation annulus and the outer physical boundary. The former costs $O(\lambda^m)$ by the cross-end core estimate; on the latter both fields are $O(\lambda^a)$ with their derivatives. This also holds for a mixed second derivative. The exterior compact region has the same bound. Since $\lambda^m=o(\eps_i^2)$ for every curvature end, these terms prove \eqref{eq:curvature-row-errors}.
\end{proof}

\subsection{The Green function terms and mixed derivatives}

\begin{proposition}\label{prop:mass-rows}
For each $i\in\mathcal J$, the first derivative and any mixed second derivative of $\mathcal A_{\rm ren}$ differ by $o(\lambda^m)$ from the corresponding derivatives of
\begin{equation}\label{eq:mass-action}
\mathcal E_{\mathcal J}=\sum_{i\in\mathcal J}\frac{a^2}{2}\eps_i^2\phi_i
-\frac m2\sum_{i\in\mathcal J}\mathfrak a_i\alpha_i^2
-m\sum_{\substack{i<j\\\{i,j\}\cap\mathcal J\ne\varnothing}}\alpha_i\alpha_jG(p_i,p_j).
\end{equation}
The estimates are uniform on the fixed compact parameter set and concern derivatives with at least one mass-end index.
\end{proposition}

\begin{proof}
First suppose all points are mass points. Put $A_1=\sum_i\alpha_iG_i$. Near $p_i$ write $A_1=\alpha_ir^{-m}+H_i$ and $C_i=\alpha_i\mathfrak a_i+\sum_{j\ne i}\alpha_jG(p_i,p_j)$. Lemma~\ref{lem:mass} and smoothness of the other Green functions give $\overline H_i=C_i+o(\lambda^a)$ and $r\overline H_i'=o(\lambda^a)$. The exact outward flux for the punctured exterior is
$$
r^{m+1}A_1(-\partial_rA_1)
=m\alpha_i^2r^{-m}+m\alpha_iH_i-\alpha_irH_i'-r^{m+1}H_iH_i'.
$$
The last term has average $o(\lambda^m)$, since $2q>m$. The normal volume error is also smaller. As $P_gA_1=0$, the quadratic exterior action finite part is $m\sum_i\alpha_iC_i/2$.

Choose smooth radial cutoffs $0\le\chi_{\rho_i}\le1$ that equal one on $B_{\rho_i}(p_i)$ and vanish outside $B_{2\rho_i}(p_i)$. They are obtained by scaling one fixed smooth radial profile, are independent of the finite bubble parameters, and satisfy $|\partial^k\chi_{\rho_i}|\le C_k\rho_i^{-k}$ for every fixed derivative order $k$ in the normal coordinates. Use the positive comparison field
$$
U_{\rm cmp}=A_1+\sum_i\chi_{\rho_i}(B_i^{\rm phys}-A_1).
$$
It is the full bubble in each smaller ball and is smooth on $M$. The mass radius chosen in Section~\ref{sec:projected} satisfies $\lambda^2\rho_i^{-n}\to0$. On a transition annulus the difference from the bubble has squared energy bounded by
$$
C\lambda^m\{\rho_i^{2q-m}+\rho_i^mL_\lambda^2+\lambda^2\rho_i^{-n}\}=o(\lambda^m).
$$
The linear bubble variation is its boundary flux, $-m\alpha_iC_i+o(\lambda^m)$; its dipole average vanishes and its remaining tail costs $O(\lambda^{m+2}\rho_i^{-n})$. To estimate the interior metric error, first take a smooth symmetric tensor $h$ with $hx=0$ and $\tr h=0$, and a complete bubble $U(x)=(2\ell/(\ell^2+|x-\xi|^2))^a$. On any ball centered at the puncture, two integrations by parts give
$$
\frac1{2\omega}\int\{-h_{jk}U_jU_k+c(\partial_j\partial_kh_{jk})U^2\}
=\frac1\omega\int h_{jk}\{(c-\tfrac12)U_jU_k+cUU_{jk}\}=0.
$$
Indeed, both boundary terms vanish because $h\nu=0$ and $\nu_j\partial_kh_{jk}=-\tr h/r=0$. Direct differentiation gives
$$
(c-\tfrac12)U_jU_k+cUU_{jk}
=-\frac{2acU^2}{\ell^2+|x-\xi|^2}\delta_{jk},
$$
since the coefficient of $(x-\xi)_j(x-\xi)_k$ is proportional to $c(2a+1)-a/2=0$. This proves exact cancellation for all bubble parameters, including their first and second derivatives. Apply it to the finite Taylor tensor $B$ in \eqref{eq:finite-normal}, and estimate $E_N$ separately. For the physical bubbles, $\ell\asymp\lambda$ and $|\xi|=O(\lambda)$ on the fixed compact parameter sets. The remaining metric terms and their two parameter derivatives are bounded by
$$
O_{C^2}\!\left(\lambda^{2q}+\lambda^m\rho_i^{2q-m}
+\lambda^N+\lambda^m\rho_i^{N-m}\right)=o(\lambda^m),
$$
because $2q>m$, $N>m$, and $\lambda/\rho_i\to0$. Exterior nonlinear energy is bounded by $C\lambda^n\rho_i^{-n}=o(\lambda^m)$. Adding the pieces gives, with $\mathfrak b_n$ the full Euclidean bubble action,
$$
\mathcal A_g(U_{\rm cmp})
=k\mathfrak b_n-\frac m2\sum_i\mathfrak a_i\alpha_i^2
-m\sum_{i<j}\alpha_i\alpha_jG(p_i,p_j)+o_{C^2}(\lambda^m).
$$
All estimates include the two parameter derivatives of the explicit bubbles and charges.

We now compare $U_{\rm cmp}$ with the filled projected field. Its difference from $U_{\rm cmp}$ is $o_{C^2}(\lambda^a)$ in $H^1$ on mass balls. For the far source correction this follows from \eqref{eq:inverse-energy} and \eqref{eq:budget-energy}; the near and Cauchy columns have the same bound by the separate energy estimate of Proposition~\ref{prop:global-inverse}. For the reference difference, the two explicit squared costs are
$$
C\lambda^m(\rho_i^{2q-m}+\rho_i^mL_\lambda^2),\qquad
C\lambda^{m+2}(\rho_i^{-n}+1),
$$
coming respectively from the truncated Green regular part and the bubble minus its monopole. The actual collar extension costs $O(\tau_i+\Theta_i+\Sigma_i+\eps_i^2)=o(\lambda^a)$. The correction $J$ to the complete Green reference has squared energy $o(\lambda^m)$ because $q>a$; the fixed-radius Green self-energy itself is already in $A_1$ and is not estimated as little-oh.

The comparison residual is $o_{C^2}(\lambda^a)$ in $H^{-1}$. More explicitly, with $s=2n/(n+2)$, its inner $L^s$ norm is at most $C\lambda^a\rho_i^{q-a}$. Its transition norms are bounded by
$$
C\lambda^a\{\rho_i^{q-a}+\rho_i^aL_\lambda+\lambda\rho_i^{-n/2}\},
$$
and its exterior power source has norm at most $C\lambda^{a+2}\rho_i^{-(n+2)/2}=o(\lambda^a)$. The required uniform relative bound follows from the explicit reference fields: on $C\lambda\le r_i\le2\rho_i$, the Green regular part divided by the bubble is at most $C\rho_i^{\min(q,m)}L_\lambda$; on $r_i\ge\rho_i$, the bubble minus its monopole divided by the Green field is $O(\lambda/\rho_i)$. The source correction is relatively $o(1)$ by Proposition~\ref{prop:projected}. On fixed scaled inner collars the normalized bubble has a positive lower bound and the filled difference tends uniformly to zero. Thus the filled difference is relatively $o(1)$ everywhere. Taylor's formula for the action now proves an $o_{C^2}(\lambda^m)$ difference. The power quadratic term is bounded by $C\|D\|_{H^1}^2$, using the uniform $L^{n/2}$ bound of $U_{\rm cmp}^{p-1}$ and this uniform relative bound. All potential negative powers in parameter derivatives are paired with a relative field factor. The positive relative Taylor formula applies for every $p>1$.

When curvature and mass ends occur together, use local estimates for derivatives with a mass-end index; the global $o(\lambda^a)$ bound on the filled difference is then unavailable. A mass derivative of the comparison field vanishes in the inner full-bubble part of a curvature ball. In the rest of that ball it is a smooth incoming field of scaled size $O(\lambda^m)$. The exterior flux derivative on that ball is $m\alpha_jD_iA_1(p_j)+o(\lambda^m)$, and its transition first variation is $-m\alpha_jD_iA_1(p_j)+o(\lambda^m)$. The errors are bounded by $C\lambda^m(\rho_j^{d_j}+\rho_j^m+\lambda/\rho_j)$. These statements persist after one arbitrary parameter derivative, since the incoming Green field and its derivative remain smooth. Thus the comparison action has exactly the mass-index rows of \eqref{eq:mass-action}; no mass is assigned to a curvature point.

Let $\beta=U_{\rm fill}-U_{\rm cmp}$ in a curvature ball and $R_0=r_0/\lambda$. The cross-end source differentiation in the preceding proof, followed by the Dirichlet inverse on $L\lambda<r<r_0$, gives
$$
|D_i\widehat\beta|_{\scnorm,1}+|D_\mu D_i\widehat\beta|_{\scnorm,1}\le C\lambda^m,
\qquad \widehat\beta(y)=\lambda^a\beta(\lambda y),\quad i\in\mathcal J.
$$
For the annular step the potential has Newton norm $CL^{-2}+C\lambda^2<1/4$ and boundary data $O(\lambda^a)$ in physical variables. A second derivative adds a source bounded by that same potential times the first derivative, since the ordinary bubble-parameter derivative is bounded relatively on this exterior annulus. The same small inverse applies.

To pair this derivative bound with $\beta$, we use the weighted first-moment estimate
\begin{equation}\label{eq:mixed-moment}
\|\widehat\beta\|_{C^1(B_{2L})}
+\int_{2L}^{R_0}r^{m-1}|\widehat\beta|_{\scnorm,1}\,dr=o(1),
\end{equation}
also after one arbitrary parameter derivative. Indeed, from \eqref{eq:radial-newton} and Fubini,
$$
\int_1^{R_0}r^{m-1}\mathcal Nf(r)\,dr
\le C\int_1^{R_0}s^{m+1}f(s)\{1+\log(R_0/s)\}\,ds.
$$
The actual first-response sources give $C\tau_i(R_0^{d_i-1}+L_\lambda)=O(\lambda)+O(\tau_iL_\lambda)$, $C\lambda^2L_\lambda^2$, and $o(1)$ for the three exterior source types. The nonlinear source pieces give $C\Sigma_iL_\lambda$, $C\rho_i^{2d_i}L_\lambda$, $CK_i^{\rm in}(\rho_i/\lambda)^mL_\lambda$, and $C\lambda^b$, all tending to zero. The reference comparison contributes $C(\rho_i^{d_i}+\rho_i^m+\lambda/\rho_i)$. This proves \eqref{eq:mixed-moment} in the centered Schauder norms as well.

Differentiate the exact action Taylor formula once in the mass direction. Every gradient cross term is bounded by $C\lambda^m$ times the left side of \eqref{eq:mixed-moment}. The inner comparison residual paired with the cross derivative is bounded by $C\lambda^m(\rho_i^{d_i}+\rho_i^m+\lambda/\rho_i)$; its mass derivative paired with $\beta$ is bounded by the same moment. Curvature and nonlinear terms satisfy these estimates using $r^2|R|\le C$ and $U_{\rm cmp}^{p-1}\le Cr^{-4}$ in scaled exterior variables. For a mixed derivative replace $\beta$ by $D_\mu\beta$, or use the displayed mixed cross bound. The mass balls and the outer compact region were already controlled. Thus the actual filled action and the comparison action have mass-index rows differing by $o(\lambda^m)$.

Finally compare the filled action with the renormalized action of the singular field. The end calculations give \eqref{eq:end-phase-action} at mass ends. At another end the cross-end estimates give
$$
\|D_iI_j\|_{\mathcal Y_{x_4,y_4}}\le C\lambda^m/\eps_j,
\qquad
\|D_i(\sigma_j/\Sigma_j)\|_{C^{2,\alpha}}\le C\lambda^m/\Sigma_j,
$$
and the same bounds after one arbitrary finite-parameter derivative. Ordinary first and second derivatives of $I_j$ and $\sigma_j/\Sigma_j$ are bounded in these spaces. Apply the joint remainder estimates specified at the beginning of this section, before substituting the common fixed point. The source and independent entrance channels are bounded by
$$
o(\eps_j^2)\left(\frac{\lambda^m}{\eps_j}+\frac{\lambda^m}{\Sigma_j}\right)
=o(\lambda^m)\left(\eps_j+\frac{\eps_j^2}{\Sigma_j}\right)=o(\lambda^m),
$$
since $\eps_j^2=O(\Sigma_j)$. The actual first-response trace channel contributes $O(\lambda^m\tau_j)$ after its linear boundary term is cancelled. In a mixed second derivative, the joint second derivative of the remainder is paired with one bounded ordinary derivative and one of the displayed cross-end derivatives; its joint first derivative is paired with the corresponding mixed cross-end derivative. The same estimates therefore apply. The actual terminal phase and its mixed derivatives are included in \eqref{eq:counterterm}. This completes \eqref{eq:mass-action} for the same projected family.
\end{proof}

The renormalized action satisfies the variational identity
\begin{equation}\label{eq:action-variation}
D_\mu\mathcal A_{\rm ren}
=\frac1\omega\sum_{i,\alpha}c_{i\alpha}
\int\Psi_{i\alpha}D_\mu U\,d\mu_g.
\end{equation}
It follows by varying the finite physical action, using \eqref{eq:projected}, cancelling the artificial conormal terms and then subtracting \eqref{eq:counterterm}. All terminal derivatives are included. The matrix of integrals on the right is invertible throughout the compact parameter set for small $\lambda$. On a single compensation annulus it tends, after the invertible parameter change \eqref{eq:parameter-change}, to the Jacobi Gram matrix. The physical scale factor is $\lambda^{n-2a-2}=1$. Off-diagonal blocks are $O(\lambda^m)$; diagonal corrections tend to zero. Therefore a critical point of this action sets every coefficient in \eqref{eq:projected} equal to zero.

\section{The finite-dimensional equations}\label{sec:reduction}

We work first in a conformal class different from the round sphere. Lemma~\ref{lem:mass} then gives $\mathfrak a_j>0$ at every $j\in\mathcal J$. The function $K_i$ in \eqref{eq:curvature-action} is positive, and
$$
T_i(\zeta)=\frac{K_i(\zeta)}{(1+|\zeta|^2)^{2d_i}}>0,\qquad T_i(\zeta)\longrightarrow0
$$
by Lemmas~\ref{lem:subcritical} and \ref{lem:critical}.

The scale equation of the leading curvature function is
$$
\frac{a^2}{2}-2d_ie^{2d_is_i}K_i(\zeta_i)=0.
$$
It has the unique solution
\begin{equation}\label{eq:curvature-scale}
s_i^0(\zeta_i)=\frac1{2d_i}\log\frac{a^2}{4d_iK_i(\zeta_i)},\qquad
\partial_{s_i}^2(\eps_i^{-2}\mathcal E_i)=-d_i a^2
\end{equation}
at that solution. Substituting the scale gives a center function equal to a constant minus $a^2\log T_i/(4d_i)$. Choose a radius $R_i$ such that
\begin{equation}\label{eq:center-boundary}
\sup_{|\zeta|=R_i}T_i(\zeta)<\tfrac12T_i(0),
\end{equation}
and let $D$ be the product of these closed center balls. All scale functions in \eqref{eq:curvature-scale} are smooth and bounded on a slightly larger product.

For a mass index write $y_j=\log(\alpha_j/\lambda^a)=as_j+a\log2$. For the fixed curvature scales put
$$
B_{jj}=m\mathfrak a_j>0,\quad B_{jk}=mG(p_j,p_k)>0\quad(j\ne k),\quad
F_j=\sum_{l\in\mathcal C}mG(p_j,p_l)2^ae^{as_l}\ge0.
$$
The mass-scale leading function, after division by $\lambda^m$, is
$$
H(y)=\sum_{j\in\mathcal J}\frac a2 y_j
-\frac12\sum_jB_{jj}e^{2y_j}
-\sum_{j<k}B_{jk}e^{y_j+y_k}-\sum_jF_je^{y_j}.
$$
For $t\ne0$ its Hessian is
\begin{equation}\label{eq:mass-concavity}
D^2H[t,t]=-2\sum_jB_{jj}e^{2y_j}t_j^2
-\sum_{j<k}B_{jk}e^{y_j+y_k}(t_j+t_k)^2
-\sum_jF_je^{y_j}t_j^2<0.
\end{equation}
The strict sign in \eqref{eq:mass-concavity} follows from the positive diagonal coefficients; the other coefficients are nonnegative. Also $H\to-\infty$ as $|y|\to\infty$. If $\max y_j\to+\infty$, a negative diagonal exponential dominates the positive linear terms; negative coordinates make nonpositive linear contributions. If $\max y_j$ is bounded and $|y|\to\infty$, at least one positive-coefficient linear term tends to $-\infty$. Thus $H$ has a unique maximum $y^0$. On the compact curvature-scale set these bounds are uniform, and \eqref{eq:mass-concavity} is uniformly negative definite at the maximum, with uniformly bounded inverse. The maximum depends smoothly on the curvature scales.

The leading mass-center function is $\sum_{j\in\mathcal J}\frac{a^2}{2}\log(1+|\zeta_j|^2)$. Its critical center is $\zeta_j=0$ and its Hessian there is $a^2I$. Regard the curvature centers $z=(\zeta_i)_{i\in\mathcal C}\in D$ as free. Arrange the remaining unknowns as curvature scales, mass centers and mass scales. The limiting normalized gradient has root
$$
x^0(z)=(s^0_{\mathcal C}(z),0,y^0(s^0_{\mathcal C}(z))).
$$
At $x^0(z)$ its derivative in $x$ is block lower triangular, with invertible diagonal blocks $\operatorname{diag}(-d_i a^2)$, $a^2I$ and $D_y^2H$. Propositions~\ref{prop:curvature-rows} and \ref{prop:mass-rows} say that the actual gradient rows, divided by their own weights $\eps_i^2$ or $\lambda^m$, converge in $C^1$ on a fixed neighborhood of the compact limiting graph. A uniform implicit function theorem, or a contraction preconditioned by the limiting inverse, gives a $C^1$ graph $x_\lambda(z)$ solving all these actual equations, with $x_\lambda\to x^0$ in $C^1$.

Let $\mathcal B_\lambda(z)=\mathcal A_{\rm ren}(x_\lambda(z),z)$. The eliminated derivatives are exactly zero, so the chain rule gives, for each curvature center,
$$
\eps_i^{-2}D_{\zeta_i}\mathcal B_\lambda
=-\frac{a^2}{4d_i}D_{\zeta_i}\log T_i(\zeta_i)+o(1)
$$
uniformly on $D$. Compactness of $D$ gives a minimum of $\mathcal B_\lambda$. If a minimum lies on a face $|\zeta_i|=R_i$, replace only that center by zero along its radial segment. Integration of its own gradient row gives
$$
\mathcal B_\lambda(z)-\mathcal B_\lambda(z^{(i\to0)})
=\frac{a^2\eps_i^2}{4d_i}\log\frac{T_i(0)}{T_i(\zeta_i)}+o(\eps_i^2)>0
$$
by \eqref{eq:center-boundary}. This contradicts minimality. The minimum is interior and all center derivatives vanish. Each boundary comparison uses the weight of that same end. If either index set is empty, remove its block. If $\mathcal C$ is empty the implicit solution already solves all equations.

All fixed parameter sets used earlier can now be chosen as a slightly larger neighborhood of the compact graphs just constructed. We make the fixed choices in the following order. Fix the true finite-normal representative and its coefficients, these compact parameter sets, the basic tail source ball and the coordinate buffers. Using the inverse without an artificial entrance, fix the auxiliary radius $r_0$, the potential cutoff $L$ and the separated-core jet bounds. Next fix $t_0$, then the exterior and additional entrance ball radii. Finally take $\lambda$ small. The far-inverse convergence and the near-block estimate in Section~\ref{sec:inverse} show that the later constants require no change in the earlier choices. These choices realize the radii $\rho_i$, the projected family and all little-oh estimates simultaneously.

At the critical parameters, \eqref{eq:action-variation} and its invertible finite matrix imply $c_{i\alpha}=0$ for every $i,\alpha$. Thus the positive projected family is an actual solution.

\begin{proof}[Proof of Theorem~\ref{thm:main}]
If $(M,[g])$ is conformally equivalent to the round sphere, apply \cite[Theorem 1]{MP} to the corresponding set of $k\ge2$ prescribed points. Pulling back its complete metric gives the required metric with scalar curvature $n(n-1)$.

Otherwise, use the finite-normal representative of Section \ref{sect2}. The classification there exhausts the points in every dimension. Lemma~\ref{lem:mass} supplies strictly positive mass for its mass points. The energy choices \eqref{eq:energy-choice} satisfy the local smallness conditions in Proposition~\ref{prop:end}, and Sections~\ref{sec:inverse}--\ref{sec:projected} produce the common positive projected family. The computations of Sections~\ref{sec:coefficients}--\ref{sec:action} and the finite-dimensional argument above give parameters for which it solves $P_gU=QU^p$.

For a fixed such $\lambda>0$, its physical end estimate is $U\ge C_i\eps_ir_i^{-a}$. The conformal length density is therefore at least $C_i'\eps_i^{1/a}/r_i$. Every curve approaching $p_i$ has length at least $C_i'\eps_i^{1/a}\int|dr_i|/r_i=\infty$. A curve that visits different ends infinitely often must repeatedly cross a fixed compact region, at a uniformly positive length cost. Since the original $M$ is compact and the number of points is finite, these cases cover all escaping curves. Thus the length metric is complete, and Hopf--Rinow gives geodesic completeness. Proposition~\ref{prop:projected} supplies smoothness and strict positivity.

Finally, if the normalized representative was $g=e^{2f}g_{\rm original}$, set $u=e^{af}U$. Conformal covariance gives the original equation and
$$
u^{4/m}g_{\rm original}=U^{4/m}g.
$$
The metrics are identical, so completeness is preserved. Since $Q/c=n(n-1)$, the scalar curvature normalization is precisely that in \eqref{eq:yamabe}.
\end{proof}

{\small
{\bf{Acknowledgments:}} {Li Ma would like to express his gratitude to Prof. R.Schoen for proposing his conjecture during his visit to Stanford University in the summer of 2017.}

\textbf{Data Availability Statement }
 Data sharing not applicable to this article as no data sets were generated or
analysed during the current study.

\textbf{Compliance with ethical standards}

\textbf{Conflict of interest} The authors declare that they have no conflict of interest.

\textbf{Use of AI tools declaration} AI was used to assist the authors with the development and checking of mathematical arguments, literature searches and reference verification, and the drafting and revision of the manuscript. The authors take full responsibility for the mathematical claims, proofs, references, and final text.
}

\end{document}